\documentclass[12pt, reqno]{amsart}
\usepackage{amssymb,amsthm,amsmath}
\usepackage[numbers,sort&compress]{natbib}
\usepackage{amssymb,amsmath}
\usepackage{amsfonts}
\usepackage{mathrsfs}
\usepackage{latexsym}
\usepackage{amssymb}
\usepackage{amsthm}
\usepackage{indentfirst}
\date{December 1, 2015}

\let\oldsection\section
\renewcommand\section{\setcounter{equation}{0}\oldsection}

\newtheorem{corollary}{Corollary}[section]
\newtheorem{theorem}{Theorem}[section]
\newtheorem{lemma}{Lemma}[section]
\newtheorem{proposition}{Proposition}[section]
\newtheorem{definition}{Definition}[section]

\newtheorem{remark}{Remark}[section]

\allowdisplaybreaks
\begin{document}

\title[Primitive equations with discontinuous initial data]{Existence and uniqueness of weak solutions to viscous primitive equations for certain class of discontinuous initial data}


\author{Jinkai~Li}
\address[Jinkai~Li]{Department of Computer Science and Applied Mathematics, Weizmann Institute of Science, Rehovot 76100, Israel}
\email{jklimath@gmail.com}

\author{Edriss~S.~Titi}
\address[Edriss~S.~Titi]{
Department of Mathematics, Texas A\&M University, 3368 TAMU, College Station, TX 77843-3368, USA. ALSO, Department of Computer Science and Applied Mathematics, Weizmann Institute of Science, Rehovot 76100, Israel.}
\email{titi@math.tamu.edu and edriss.titi@weizmann.ac.il}



\begin{abstract}
We establish some conditional uniqueness of weak solutions to the
viscous primitive equations, and as an application, we prove the
global existence and uniqueness of weak solutions, with the initial
data taken as small
$L^\infty$ perturbations of functions in the space
$X=\left\{v\in (L^6(\Omega))^2|\partial_zv\in
(L^2(\Omega))^2\right\}$; in particular, the initial data are allowed
to be discontinuous. Our result generalizes in a uniform way the
result on the uniqueness of weak solutions with continuous initial
data and that of the so-called $z$-weak solutions.
\end{abstract}

\maketitle

\allowdisplaybreaks

\section{Introduction}
\label{sec1}

The primitive equations
are derived from the Boussinesq system of incompressible flow under
the hydrostatic balance assumption, by taking the zero singular perturbation limit of the small aspect ratio. This singular perturbation limit of small aspect ratio can be rigorously justified, see, Az\'erad--Guill\'en \cite{AZGU} and Li--Titi \cite{LITITIHYDRO}. The primitive equations play a fundamental role for weather prediction models, see, e.g., the books by Lewandowski
\cite{LEWAN}, Majda \cite{MAJDA}, Pedlosky \cite{PED}, Vallis
\cite{VALLIS}, and Washington--Parkinson \cite{WP}.

In this paper, we consider the following primitive equations (without considering the coupling to the temperature equation):
\begin{eqnarray}
&\partial_tv+(v\cdot\nabla_H)v+w\partial_zv+\nabla_Hp(\textbf{x}^H,t)-\Delta v+f_0k\times
v=0,\label{MAIN1}\\
&\nabla_H\cdot v+\partial_zw=0,\label{MAIN2}
\end{eqnarray}
where the horizontal velocity $v=(v^1,v^2)$, the vertical velocity
$w$ and the pressure $p$ are the unknowns, and $f_0$ is the Coriolis
parameter. Note that though (\ref{MAIN1})--(\ref{MAIN2}) is a
three-dimensional system, the pressure $p$ depends only on two spatial
variables, and there is no dynamical equation for the vertical velocity
$w$. We use $\textbf{x}=(x^1, x^2, z)$ to denote the spatial variable,
$\textbf{x}^H=(x^1, x^2)$ the horizontal spatial variable, and
$\nabla_H=(\partial_1,\partial_2)$ the horizontal gradient. We complement  system (\ref{MAIN1})--(\ref{MAIN2}) with initial and boundary conditions which  will be discussed in  details below. Note that the primitive equations considered here, i.e.,
(\ref{MAIN1})--(\ref{MAIN2}), do not include the coupling with
the temperature equation;
however, one can adopt the same approach presented in this paper to
establish the corresponding results for the full primitive equations coupled
with temperature equation, for as long as the temperature equation has
full diffusivity. For the sake of simplicity,
we consider here the primitive equations without the coupling with the temperature equation,
i.e., system (\ref{MAIN1})--(\ref{MAIN2}).

Since the pioneer works by Lions--Temam--Wang \cite{LTW92A,LTW92B,LTW95}
in the 1990s, there has been a lot of literatures on the mathematical
studies of the primitive equations. The global existence of
weak solutions has been proven long time ago by Lions--Temam--Wang
\cite{LTW92A,LTW92B,LTW95} in the 1990s; however, the uniqueness of
weak solutions in the general case is still an open question, even
for the two-dimensional case. This is different from the
incompressible Navier-Stokes equations, since it is well-known that
the weak solution to the two-dimensional incompressible
Navier-Stokes equations is unique (see, e.g., Ladyzhenskaya \cite{LADINBOOK}, Temam \cite{TEMNSBOOK} and Constantin--Foias \cite{CONFOINSBOOK}).
The main obstacle of proving
the uniqueness of weak solutions to the two-dimensional
primitive equations is the absence of the dynamical equation for
the vertical velocity. In fact, the vertical velocity can only
be recovered from the horizontal velocity through the
incompressibility condition, and as a result, there is one
derivative loss for the horizontal velocity.
Though the general result on the uniqueness of weak solutions to
the primitive equations is still unknown, some particular cases
have been solved, see Bresch et al \cite{BGMR03}, Petcu--Temam--Ziane \cite{PTZ09} and Tachim Madjo \cite{TACHIM} for the case of the
so-called $z$-weak solutions, i.e., the weak solutions with initial
data in $X\cap\mathcal H$ (the spaces $X$ and $\mathcal H$ are given,
below, by (\ref{X}) and (\ref{H}), respectively),
and Kukavica--Pei--Rusin--Ziane \cite{KPRZ} for the case of
weak solutions with continuous initial data.
In the context of strong solutions, the local well-posedness
was established in Guill\'en-Gonz\'alez et al \cite{GMR01}. Remarkably,
the strong solutions to the primitive equations have already been
proven to be global for
both 2D and 3D cases, see Cao--Titi \cite{CAOTITI2,CAOTITI3},
Kobelkov \cite{KOB06} and Kukavica--Ziane \cite{KZ07A,KZ07B}, see
also Cao--Li--Titi \cite{CAOLITITI1,CAOLITITI2,CAOLITITI3,CAOLITITI4,CAOLITITI5}
for some recent progress in the direction of global well-posedness
of strong solutions to the systems with partial viscosities or
diffusivity, as well as Hieber--Kashiwabara \cite{HIEKAS}
for some progress towards relaxing the smoothness on the initial
data but still for the system with full viscosities and diffusivity,
by using the semigroup method. Notably, smooth solutions to
the inviscid primitive equations, with or without coupling to the temperature equation, have
been shown by Cao et al.\,\cite{CINT} and Wong \cite{TKW} to blow up in
finite time. For the results on the local well-posedness with monotone
or analytic initial data of the
inviscid primitive equations (which are also called inviscid Prandtl equations or hydrostatic Euler equations), i.e.,
system (\ref{MAIN1})--(\ref{MAIN2}), without the Lapalacian term $\Delta v$, see, e.g., Brenier \cite{BRE}, Masmoudi--Wong \cite{MASWON}, Kukavica--Temam--Vicol--Ziane \cite{KTVZ}, and the references therein.

Another
remarkable difference between the incompressible Navier-Stokes equations
and the primitive equations is their well-posedness theories with $L^p$
initial data. It is well-known that, for any $L^p$ initial data, with
$d\leq p\leq\infty$, there is a unique local mild solution to the
$d$-dimensional incompressible Navier-Stokes equations, see, Kato \cite{KATO} and Giga \cite{GIGA86,GIGA99};
however, for the primitive equations, though the $L^p$ norms of the
weak solutions keep finite up to any finite time, as long as the
initial data belong to $L^p$ (one can apply Proposition 3.1 in Cao--Li--Titi \cite{CAOLITITI3} to achieve this fact with the help of some regularization procedure), it is
still an open question to show the uniqueness of weak solutions to
the primitive equations with $L^p$ initial data. One of the aims of this paper is to partly answer the question of the
existence and uniqueness of weak solutions to the primitive equations, with initial data in $L^\infty$. It will be shown, as a consequence of our result, that the primitive equations possess a unique global weak solution, with initial data in $L^\infty$, as long
as the discontinuity of the initial data is sufficiently small.

We consider system (\ref{MAIN1})--(\ref{MAIN2}) in the three-dimensional horizontal layer $\Omega_0=M\times(-h, 0)$, confined between the horizontal walls $z=-h$ and $z=0$, subject to periodic boundary conditions with respect to the horizontal variables $\textbf{x}^H=(x^1,x^2)$ with basic fundamental periodic domain  $M=(0,1)\times(0,1) \subset \mathbb{R}^2$. We also suppose that the flow is stress-free at, and tangential to, the solid boundaries $z=-h$ and $z=0$. In other words,
we complement system (\ref{MAIN1})--(\ref{MAIN2}) with the following boundary and initial conditions
\begin{eqnarray}
  &v, w\text{ and }p\text{ are periodic in }\textbf{x}^H, \label{bc1'}\\
  &w|_{z=-h,0}=0,\quad \partial_zv|_{z=-h,0}=0, \label{bc2'}\\
  &v|_{t=0}=v_0. \label{ic'}
\end{eqnarray}
Extend the unknowns $v$ and $w$ evenly and oddly, respectively,
with respect to $z$, to be defined on the larger spatial domain
$\Omega=M\times(-h,h)$, then the extended unknowns $v$ and $w$ are
periodic, and are even and odd, respectively, in $z$. After such kind
extension of the unknowns, the boundary and initial conditions
(\ref{bc1'})--(\ref{ic'}) are equivalent to the following periodic boundary and initial
conditions
\begin{eqnarray}
  &v, w\text{ and }p\text{ are periodic in }x^1, x^2\text{ and }z,\label{BC1}\\
  &v \text{ and } w \text{  are even and odd in } z, \text{respectively,}\label{BC2}\\
  &v|_{t=0}=v_0. \label{IC}
\end{eqnarray}

Because of the equivalence of the above two kinds of boundary and
initial conditions, we consider, throughout this paper, the
boundary and initial conditions (\ref{BC1})--(\ref{IC}), in other
words, we consider system (\ref{MAIN1})--(\ref{MAIN2}), subject
to (\ref{BC1})--(\ref{IC}).
Note that condition (\ref{BC2}) is a symmetry condition, which is
preserved by system (\ref{MAIN1})--(\ref{MAIN2}), that is, if a smooth
solution to system (\ref{MAIN1})--(\ref{MAIN2}) exists and is unique, then it must satisfy the symmetry
condition (\ref{BC2}), as long as it is initially satisfied. We also
note that since there is no dynamical equation for the vertical
velocity, no initial condition is imposed on $w$. In fact, it is
uniquely determined by the horizontal velocity, through the
incompressibility condition (\ref{MAIN2}) and the boundary conditions
on $w$.

Note that the vertical velocity $w$ can be uniquely expressed in terms of the horizontal velocity $v$, through the incompressibility condition (\ref{MAIN2}) and the symmetry condition (\ref{BC2}), as
$$
w(x,y,z,t)=-\nabla_H\cdot\left(\int_{-h}^z v(x,y,\xi,t)d\xi\right).
$$
As a result, system (\ref{MAIN1})--(\ref{IC}) is equivalent to the following one
\begin{eqnarray}
&\partial_tv+(v\cdot\nabla_H)v+w\partial_zv+\nabla_Hp(\textbf{x}^H,t)-\Delta v+f_0k\times
v=0,\label{main1}\\
&\nabla_H\cdot\left(\int_{-h}^h v(x,y,z,t)dz\right)=0,\label{main1'}\\
&w(x,y,z,t)=-\nabla_H\cdot\left(\int_{-h}^z v(x,y,\xi,t)d\xi\right),\label{main2}
\end{eqnarray}
subject to the boundary and initial conditions
\begin{eqnarray}
  &p \mbox{ is periodic in }\textbf{x}^H; \text{ and }v\mbox{ is periodic in }\textbf{x}^H \text{ and }z, \mbox{ and is even in }z,\label{bc1}\\
  &v|_{t=0}=v_0. \label{ic}
\end{eqnarray}

Let us introduce some necessary notations, and give the definition of
weak solutions. We denote by
$$
C_\text{per}(\Omega)=\{f\in C(\mathbb{R}^3)|f\text{ is periodic in }x^1,x^2\mbox{ and }z, \mbox{ with basic periodic domain } \Omega\},
$$
and for any positive integer $m$, we set
$$
C_\text{per}^m(\Omega)=\{f|\nabla^\alpha f\in C_\text{per}(\Omega), 0\leq|\alpha|\leq m\}.
$$
For positive integer $m$ and number $q\in[1,\infty]$, we use
$W_\text{per}^{m,q}(\Omega)$ to denote the space of
the closure of $C_\text{per}^m(\Omega)$ in $W^{m,q}(\Omega)$, and
denote by $W^{-m,q'}_\text{per}(\Omega)$ its dual space,
when $q\in[1,\infty)$. One
can easily verify, with the help of the standard modifier, that
$$
W_\text{per}^{m,q}(\Omega)=\{f\in W^{m,q}(\Omega)|\tilde f\in W^{m,q}_\text{loc}(\mathbb R^3)\},
$$
where $\tilde f$ is the periodic extension of $f$ to the whole space.
In case that $q=2$, we always use $H^m_\text{per}(\Omega)$ instead of
$W^{m,2}_{\text{per}}(\Omega)$. One may also define the periodic
$L^q(\Omega)$ space as the closure of $C_\text{per}(\Omega)$ in
$L^q(\Omega)$; however, one can easily check that this space coincides
with $L^q(\Omega)$. For simplicity, we use the same notations to
denote a space and its product spaces, that is for a space $Z$,
and a positive integer $N$, we still use $Z$ to denote its $N$-product
space $Z^N$. We always use $\|u\|_p$ to denote the $L^p$ norm of $u$.
The following two spaces $X$ and $\mathcal H$ will be used throughout this paper
\begin{equation}\label{X}
X=\{v\in L^6(\Omega)|v\text{ is periodic in }z, \mbox{ and }\partial_zv\in L^2(\Omega)\},
\end{equation}
and
\begin{equation}\label{H}
\mathcal H=\left\{ v\in L^2(\Omega)\bigg|v\mbox{ is periodic in }\textbf{x}^H, \mbox{ even in }z, \mbox{ and } \nabla_H\cdot \left(\int_{-h}^hv(\textbf{x}^H,z)dz\right)=0\right\}.
\end{equation}

We state the definition of weak solutions to system (\ref{main1})--(\ref{ic}) as follows.

\begin{definition}
  \label{def1.1}
Given a function $v_0\in\mathcal H$. A function $v$ is called a global weak solution to system (\ref{main1})--(\ref{ic}), if the following hold:

(i) $v\in C([0,\infty);L^2_w(\Omega))\cap L^2_{\text{loc}}([0,\infty); H^1_{per}(\Omega)\cap\mathcal H)$, where $L^2_w$ stands for the $L^2$ space equipped with the weak topology.

(ii) For any compactly supported in time function $\varphi\in C([0,\infty); C_\text{per}^1(\Omega)\cap\mathcal H))\cap C^1([0,\infty); C_\text{per}(\Omega))$, the following equality holds
$$
\int_0^\infty\int_\Omega[-v\partial_t\varphi+(v\cdot\nabla_H) v\cdot\varphi+w\partial_zv\cdot\varphi+\nabla v:\nabla\varphi]d\textbf{x}dt=\int_\Omega v_0(\textbf{x})\varphi(\textbf{x},0)d\textbf{x},
$$
where the vertical velocity $w$ is given by (\ref{main2}).

(iii) The following differential inequality holds
$$
\frac12\frac{d}{dt}\|v\|_2^2(t)+\|\nabla v\|_2^2(t)\leq0,\quad\mbox{in }\mathcal D'((0,\infty));
$$

(iv) The following energy inequality holds
$$
\frac12\|v\|_2^2(t)+\int_0^t\|\nabla v\|_2^2(\tau)d\tau\leq\frac12\|v_0\|_2^2,
$$
for a.e.~$t\in(0,\infty)$.
\end{definition}

For any initial data $v_0\in\mathcal H$, following the arguments in
\cite{LTW92A,LTW92B,LTW95}, there is a global weak solution to system
(\ref{main1})--(\ref{ic}); however, as we mentioned before, it is still
an open question to show the uniqueness of weak solutions to the
primitive equations. Nevertheless, we can prove the following theorem
on the conditional uniqueness of weak solutions to the primitive
equations.

%
%


\begin{theorem}
  \label{thmwsu}
  Let $v$ be a global weak solution to system (\ref{main1})--(\ref{ic}). Suppose that there is a positive time $T_v$, such that
  \begin{eqnarray*}
  v(\textbf{x},t)=\bar v(\textbf{x},t)+V(\textbf{x},t),\quad \textbf{x}\in\Omega, t\in(0,T_v),\\
  \partial_z\bar v\in L^\infty(0,T_v; L^2(\Omega))\cap L^2(0,T_v; H^1_\text{per}(\Omega)),\quad V\in L^\infty(\Omega\times(0, T_v)).
  \end{eqnarray*}

  Then, there is a positive constant $\varepsilon_0$ depending only on $h$, such that $v$ is the unique global weak solution to system (\ref{main1})--(\ref{ic}), with the same initial data as $v$, provided
  $$
  \sup_{0<t<T_v'}\|V\|_\infty\leq\varepsilon_0,
  $$
  for some $T_v'\in(0,T_v)$.
\end{theorem}

To prove Theorem \ref{thmwsu}, we first show that any weak solution
to the primitive equations is smooth away from initial time, see Corollary \ref{corinterior}, below,
which can be intuitively seen by noticing that weak solution has
$H^1_\text{per}$ regularity, at almost any time away from the initial time, and
recalling that, for any $H^1_\text{per}$ initial data, there is a unique global
strong solution to the primitive equations; however, in order to
rigorously prove this fact, we need the the weak-strong uniqueness
result for the primitive equations, see Proposition \ref{apppropwsu},
below, where we adopt the idea of Serrin \cite{SERRIN}. Since weak
solutions are smooth away from the initial time, one can perform the
energy estimates to the difference system between two weak solutions,
on any finite time interval away from the initial time. Next, by using
the decomposition stated in the theorem, and handling the nonlinear
terms involving each parts in their own ways, we then achieve the
uniqueness.

It should be pointed out that one can not expect that all weak
solutions to the primitive equations, with general initial data in
$\mathcal H$, have the decomposition as stated in the above theorem.
In fact, by the weakly lower semi-continuity of the norms, in order
to have such decomposition, it is necessary to require that the
initial data $v_0$ has the decomposition $v_0=\bar v_0+V_0$, with
$\partial_z\bar v_0\in L^2$ and $V_0\in L^\infty$. Observing this, it is
necessary to state the following theorem on the global existence and
uniqueness of weak solutions to the primitive equations with such kind
initial data.

\begin{theorem}\label{thmain}
Suppose that the initial data $v_0=\bar v_0+V_0$, with $\bar v_0\in
X\cap\mathcal H$ and $V_0\in L^\infty(\Omega)\cap\mathcal H$. Then the
following hold:

(i) There is a global weak solution $v$ to system (\ref{main1})--(\ref{ic}),
such that
\begin{eqnarray*}
v(\textbf{x},t)=\bar v(\textbf{x},t)+V(\textbf{x},t), \quad\textbf{x}\in\Omega, t\in(0,\infty),\\
\partial_z\bar v\in L_{loc}^\infty([0,\infty);L^2(\Omega))\cap L_{loc}^2([0,\infty); H^1_\text{per}(\Omega)),\\
\sup_{0<s<t}\|V\|_\infty(s)\leq \mu(t)\|V_0\|_\infty,\quad t\in(0,\infty),
\end{eqnarray*}
where
$$
\mu(t)=C_0(1+\|v_0\|_4)^{40}(t+1)^2\exp\{C_0e^{2t}(t+1)(1+\|v_0\|_4)^4\},
$$
for some positive constant $C_0$ depending only on $h$;

(ii) Let $\varepsilon_0$ be the positive constant in Theorem \ref{thmwsu}, then the above weak solution is unique, provided
$\mu(0)\|V_0\|_\infty\leq\frac{\varepsilon_0}{2}$.
\end{theorem}

The uniqueness part of Theorem \ref{thmain} is a direct consequence
of the estimate in (i) and Theorem \ref{thmwsu}.
So the key ingredient of the proof of
Theorem \ref{thmain} is to find the required decomposition. Note that
in \cite{KPRZ}, the authors decompose the weak solution $v$ into a regular
part $\bar v$, which is a strong solution to the primitive equations,
with $H^2$ initial data, and a small perturbation bounded part $V$, which satisfies
a nonlinear system, with small initial data in $L^\infty$. Noticing that
the initial data $v_0$ considered in this paper can be discontinuous, and
recalling the Sobolev embedding relation $H_\text{per}^2(\Omega)\hookrightarrow
C_\text{per}(\Omega)$, hence $v_0$ cannot be approximated in the $L^\infty$ norm by
$H^2$ functions, and consequently the decomposition used in \cite{KPRZ} does not apply to our case.
One may still use the same decomposition as in \cite{KPRZ}, but change the initial data
for $\bar v$ and $V$ in $X$ and $L^\infty$, respectively; however,
if by taking this approach, we are not able to obtain the $L^\infty$
estimates on $V$, because the $L^\infty_t(H^2_\textbf{x})$ estimate on $\bar v$
plays an essential role in deriving such estimates on $V$, and it
is obvious
that primitive equations with initial data in $X$ does not necessary
provide the required $L^\infty_t(H^2_\textbf{x})$ estimate on $\bar v$. We also note that the
arguments used in \cite{BGMR03,PTZ09,TACHIM} require the regularity of
the kind $\partial_zv\in L^\infty_t(L^2_\textbf{x})\cap L^2_t(H^1_\textbf{x})$ for the $z$-weak
solutions, and thus still cannot be applied to our case.

To obtain the desired decomposition on $v$, we use a different kind
of decomposition from that used in \cite{KPRZ}. Indeed (ignoring
the regularization procedure to justify the arguments), we decompose
the weak solution $v$ into a ``regular" part $\bar v$, which is the
unique solution to a linearized primitive equations (i.e.~system
(\ref{3.1})--(\ref{3.2}), below), with initial data in $X$, and the
remaining part $V$, which satisfies the same system as that for $\bar
v$, but with initial data in $L^\infty$. The key observation is that both the $L^\infty$ and the $X$ regularities are
preserved by the linearized primitive equations, i.e.~systems
(\ref{3.5})--(\ref{3.6}) and (\ref{3.1})--(\ref{3.2}), below.
Indeed, by making use of such kind of decomposition, we are able
to weaken the assumption on the initial data, without destroying the
uniqueness of the weak solutions. We point out that, compared with the
decomposition introduced in \cite{KPRZ}, in our decomposition systems,
i.e.~systems
(\ref{3.5})--(\ref{3.6}) and (\ref{3.1})--(\ref{3.2}), there are no nonlinear terms and
no cross terms between $\bar v$ and $V$, and moreover, the systems for $\bar v$ and $V$ are
decoupled and are independent of each other. As will be shown later, see Proposition
\ref{prop3.3} and Proposition \ref{prop3.4}, below, the estimates for
$\bar v$ and $V$ in $X$ and $L^\infty$, respectively, are global in
time, and thus no local in time arguments are required for deriving
these estimates.

As an application of Theorem \ref{thmain}, we have the following result, which generalizes the results in \cite{BGMR03,TACHIM,PTZ09,KPRZ}.

\begin{corollary}\label{thm1.2}
For any $v_0\in X\cap\mathcal H$ (or $v_0\in C_\text{per}(\Omega)\cap\mathcal H$), there is a constant $\sigma_0$, depending only on $h$ and the upper bound of $\|v_0\|_4$, such that for any $ \mathscr V_0=v_0+V_0$, with $V_0\in L^\infty(\Omega)\cap\mathcal H$ and $\|V_0\|_\infty\leq\sigma_0$, system (\ref{main1})--(\ref{ic}) with initial data $\mathscr V_0$ has a unique weak solution, which has the regularities stated in Theorem \ref{thmain}.
\end{corollary}

\begin{proof}
\textbf{(i) The case $v_0\in X\cap\mathcal H$.} Let $\varepsilon_0$ be the constant in Theorem \ref{thmwsu}.
Recalling that $\mathscr V_0=v_0+V_0$, by the triangle inequality for the norms and the H\"older inequality, one can easily check that
\begin{align*}
  C_0(1+\|\mathscr V_0\|_4)^{40}\exp\{C_0 (1+\|\mathscr V_0\|_4)^4\}\|V_0\|_\infty
  \leq\rho(\|V_0\|_\infty),
\end{align*}
where
$$
\rho(s)=C_0(1+\| v_0\|_4+(2h)^{\frac{1}{4}}s)^{40}\exp\{C_0(1+\| v_0\|_4+(2h)^{\frac{1}{4}}s)^4\}s.
$$
Noticing that $\rho$ is a continuous function, with $\rho(0)=0$, and $C_0$ is a positive constant depending only on $h$, there is a positive constant $\sigma_0$ depending only on $h$ and the upper bound of $\|v_0\|_4$, such that $\rho(s)\leq\varepsilon_0/2,$ for any $s\in[0,\sigma_0]$. Thus, for any $V_0\in L^\infty(\Omega)\cap\mathcal H$, with $\|V_0\|_\infty\leq\sigma_0$, we have
$$
C_0(1+\|\mathscr V_0\|_4)^{40}\exp\{C_0 (1+\|\mathscr V_0\|_4)^4\}\|V_0\|_\infty\leq \rho(\|V_0\|_\infty)\leq\varepsilon_0/2,
$$
and consequently, the conclusion follows by (ii) of Theorem \ref{thmain}, by viewing $(\mathscr V_0, v_0, V_0)$ as the $(v_0, \bar v_0, V_0)$ in Theorem \ref{thmain}.

\textbf{(ii) The case $v_0\in C_\text{per}(\Omega)\cap\mathcal H$.} Let $\sigma_0$ be the same positive constant as in (i). We are going to show that the conclusion holds in this case for the new $\sigma_0^*:=\frac{\sigma_0}{2}$. Suppose that
$$
\|V_0\|_\infty\leq\sigma_0^*.
$$
Let $j_\eta$ be a standard mollifier,
i.e.\,$j_\eta(\textbf{x})=\frac{1}{\eta^3}j(\frac{\textbf{x}}{\eta})$,
with $0\leq j\in C_0^\infty(\mathbb R^3)$ and
$\int_{\mathbb R^3}jd\textbf{x}=1$. Since $v_0\in C_\text{per}(\Omega)$, we
have
$$
\tilde v_0*j_\eta\rightarrow v_0,\quad\mbox{ in }C_\text{per}(\Omega), \mbox{ as }\eta\rightarrow 0^+,
$$
where $\tilde v_0$ is the periodic extension of $v_0$ to the whole
space. Choose $\eta_0\leq\min\{1,2h\}$ to be small enough, such that $\|v_0-\tilde v_0*j_{\eta_0}\|_\infty\leq\sigma_0^*$, and set $\bar v_0=\tilde v_0*j_{\eta_0}$. We have
$$
\mathscr V_0=v_0+V_0=\bar v_0+(v_0-\tilde v_0*j_{\eta_0})+V_0=:\bar v_0+\tilde V_0,
$$
where $\tilde V_0:=(v_0-\tilde v_0*j_{\eta_0})+V_0$.
Recalling that $\|V_0\|_\infty\leq\sigma_0^*,$ it follows that
$$
\|\tilde V_0\|_\infty\leq\|v_0-\tilde v_0*j_{\eta_0}\|_\infty +\|V_0\|_\infty
\leq\sigma_0^*+\sigma_0^*=\sigma_0.
$$
Therefore, the initial data $\mathscr V_0$ can be decomposed
as $\mathscr V_0=\bar v_0+\tilde V_0$, with
$\bar v_0=j_{\eta_0}*\tilde v_0\in X\cap\mathcal H$, $\|\bar v_0\|_4\leq9\|v_0\|_4$, and
$\|\tilde V_0\|_\infty\leq\sigma_0$, where $\sigma_0$ is the
same constant as in (i), which depends only on
$\|v_0\|_4$. By viewing $(\mathscr V_0, \bar v_0, \tilde V_0)$ as the $(\mathscr V_0, v_0, V_0)$ in case (i), it is clear that the assumptions in case (i) are fulfilled, and thus there is a unique weak solution
to the primitive equations with initial data $\mathscr V_0$.
\end{proof}

\begin{remark}
Given a constant vector $a=(a^1, a^2)$, and two positive numbers
$\delta$ and $\eta$, with $\eta\in(0,h)$. Set $v_0=a|z|^\delta$,
$V_0=\sigma\chi_{(-\eta,\eta)}(z)$, and $\mathscr V_0=v_0+V_0$, for $z\in(-h,h)$, with
$\sigma=(\sigma^1,\sigma^2)$, where $\chi_{(-\eta, \eta)}(z)$ is the
characteristic function of
the interval $(-h, h)$. Extend $v_0$ and $V_0$, and
consequently $\mathscr V_0$, periodically to the
whole space, and still use the same notations to denote the
extensions. Then, one can easily check that $v_0\in
C_\text{per}(\Omega)\cap\mathcal H$ and $V_0\in
L^\infty(\Omega)\cap\mathcal H$.
By Corollary \ref{thm1.2},
there is a positive constant
$\varepsilon_0=\varepsilon_0(a,\delta,\eta,h)$, such that,
for any $\sigma=(\sigma^1,\sigma^2)$, with
$0<|\sigma|\leq\varepsilon_0$, system (\ref{main1})--(\ref{ic})
has a unique weak solution, with initial data $\mathscr V_0$.
Note that
$$
\mathscr V_0=a|z|^\delta+\sigma\chi_{(-\eta,\eta)}(z), \quad z\in(-h,h).
$$
One can easily verify that $\mathscr V_0$ lies neither
in $X\cap\mathcal H$ nor
in $C_\text{per}(\Omega)\cap\mathcal H$, and thus the results
established in \cite{BGMR03,TACHIM,PTZ09,KPRZ} cannot be applied to
prove the uniqueness of weak solutions with such kind of initial data.
\end{remark}

The rest of this paper is arranged as follows: in section \ref{sec2},
we collect some preliminary results; in section
\ref{secinterior}, as a preparation of proving Theorem \ref{thmwsu},
we show the regularities, away from the initial time,
of weak solutions to the primitive equations;
the proof of Theorem \ref{thmwsu} is given in section
\ref{secwsu}; as a preparation of proving Theorem \ref{thmain}, we
derive some relevant a priori estimates of the solutions to the
primitive equations, with smooth initial data; the proof of Theorem
\ref{thmain} is given in the last section, section \ref{sec4}.

\section{Preliminaries}
\setcounter{tocdepth}{1}
\label{sec2}
In this section, we collect some preliminary results which will be used in the rest of this paper. We start with the following lemma, which can be proven in the same way as Proposition 2.2 in Cao--Titi \cite{CAOTITI1}, and thus we omit the proof of it here.

\begin{lemma}\label{lad}
The following inequalities hold
\begin{align*}
&\int_M\left(\int_{-h}^h|\phi(\textbf{x}^H,z) |dz\right)\left(\int_{-h}^h|\varphi(\textbf{x}^H,z) \psi(\textbf{x}^H,z) |dz\right)d\textbf{x}^H\\
\leq&C\|\phi\|_2\|\varphi\|_2^{\frac{1}{2}}\left(\|\varphi\|_2 +\|\nabla_H\varphi\|_{2}
\right)^{\frac{1}{2}}\|\psi\|_2^{\frac{1}{2}}\left(
\|\psi\|_2+\|\nabla_H\psi\|_2 \right)^{\frac{1}{2}},
\end{align*}
and
\begin{align*}
&\int_M\left(\int_{-h}^h|\phi(\textbf{x}^H,z) |dz\right)\left(\int_{-h}^h|\varphi(\textbf{x}^H,z) \psi(\textbf{x}^H,z) |dz\right)d\textbf{x}^H\\
\leq&C\|\phi\|_2^{\frac{1}{2}}\left(
\|\phi\|_2+\|\nabla_H\phi\|_2 \right)^{\frac{1}{2}} \|\varphi\|_2^{\frac{1}{2}}\left(\|\varphi\|_2 +\|\nabla_H\varphi\|_{2}
\right)^{\frac{1}{2}}\|\psi\|_2,
\end{align*}
for every $\phi,\varphi$ and $\psi$, such that the quantities on the right-hand sides make sense and are finite.
\end{lemma}

The next lemma will be used in Proposition \ref{prop3.3} to establish $L^\infty$ estimates for $V_\varepsilon$.

\begin{lemma}
\label{iterate}
Let $M_0\geq2$ and $\delta_0>0$ be two constants. Suppose that the sequence $\{A_k\}_{k=1}^\infty$, with $A_k\geq0$, satisfies
$$
A_1\leq M_0\delta_0^2, \quad A_{k+1}\leq M_0\delta_0^{2^{k+1}}+M_0^kA_k^2,\quad \mbox{for }k=1,2,\cdots.
$$
Then one has
$$
A_k\leq M_0^{-(k+2)}(M_0^4\delta_0^2)^{2^{k-1}},\quad k=1,2,\cdots,
$$
and in particular $A_k\leq (M_0^4\delta_0^2)^{2^{k-1}}$, for $k=1,2,\cdots.$
\end{lemma}

\begin{proof}
Define $a_k=M_0^{-(k+2)}(M_0^4\delta_0^2)^{2^{k-1}}$, for $k=1,2,\cdots$. Then, it suffices to prove that $A_k\leq a_k$.
For $k=1$, it follows from the assumption that
$$
a_1=M_0^{-3}M_0^4\delta_0^2=M_0\delta_0^2\geq A_1.
$$
Suppose that $A_k\leq a_k$ holds, for some $k$. Then, by assumption, we have
\begin{align*}
  A_{k+1}\leq&M_0\delta_0^{2^{k+1}}+M_0^ka_k^2=M_0\delta_0^{2^{k+1}}+M_0^{-(k+4)}(M_0^4\delta_0^2)^{2^k}\\
  =&M_0\delta_0^{2^{k+1}}+M_0^{-1}a_{k+1}\leq M_0\delta_0^{2^{k+1}}+\frac{a_{k+1}}{2}.
\end{align*}
We will show that
$$
M_0\delta_0^{{2^{k+1}}}\leq\frac{a_{k+1}}{2},
$$
and as a consequence of the above, we obtain
$$
A_{k+1}\leq\frac{a_{k+1}}{2}+\frac{a_{k+1}}{2}=a_{k+1}.
$$
Therefore, by induction, the conclusion holds.

We now prove that $M_0\delta_0^{{2^{k+1}}}\leq\frac{a_{k+1}}{2},$ for $k=1,2,\cdots.$ Recalling the definition of $a_k$, it follows
$$
a_{k+1}=M_0^{-(k+3)}(M_0^4\delta_0^2)^{2^k}=M_0^{4\times 2^k-(k+3)}\delta_0^{2^{k+1}}.
$$
Using binomial expansion, we have
$$
4\times 2^k=4(1+1)^k=4\sum_{j=0}^k\begin{pmatrix}
                                    k \\
                                    j \\
                                  \end{pmatrix}
\geq4(1+k),
$$
and thus
$$
4\times 2^k-(k+3)\geq 4k+4-k-3=3k+1\geq 4.
$$
Therefore, we have
$$
a_{k+1}=M_0^{4\times 2^k-(k+3)}\delta_0^{2^{k+1}}\geq M_0^4\delta_0^{2^{k+1}}\geq 2^3M_0\delta_0^{2^{k+1}}\geq 2M_0\delta_0^{2^{k+1}}.
$$
This completes the proof.
\end{proof}

We also need the following Aubin--Lions lemma.

\begin{lemma}[cf. Simon \cite{Simon} Corollary 4]\label{AL}
Let $T\in(0,\infty)$ be given. Assume that $X, B$ and $Y$ are three Banach spaces, with $X\hookrightarrow\hookrightarrow B\hookrightarrow Y.$ Then it holds that

(i) If $F$ is a bounded subset of $L^p(0, T; X)$, where $1\leq p<\infty$, and $\frac{\partial F}{\partial t}=\left\{\frac{\partial f}{\partial t}|f\in F\right\}$ is bounded in $L^1(0, T; Y)$, then $F$ is relatively compact in $L^p(0, T; B)$;

(ii) If $F$ is bounded in $L^\infty(0, T; X)$ and $\frac{\partial F}{\partial t}$ is bounded in $L^r(0, T; Y)$, where $r>1$, then $F$ is relatively compact in $C([0, T]; B)$.
\end{lemma}

\section{Regularities of weak solutions for positive time}
\label{secinterior}
In this section, we prove that the global weak solution to system (\ref{main1})--(\ref{ic}) is smooth away from the initial time.
To prove this fact, we will use
the weak-strong uniqueness result for the primitive equations.
Therefore, Let us recall the definition of strong solutions to the
primitive equations as follows:

\begin{definition}
  Given a positive time $T\in(0,\infty)$ and the initial data $v_0\in H^1_{\text{per}}(\Omega)\cap \mathcal H$. A function $v$ is called a strong solution to system (\ref{main1})--(\ref{ic}), on $\Omega\times(0,T)$, if
  $$
  v\in C([0,T];H^1_\text{per}(\Omega)\cap\mathcal H)\cap L^2(0,T; H^2_{\text{per}}(\Omega)), \quad\partial_tv\in L^2(0,T; L^2(\Omega)),
  $$
  satisfies equation (\ref{main1}) pointwisely, a.e.\,in $\Omega\times(0,T)$, and fulfills the initial condition (\ref{ic}), with $w$ given by equation (\ref{main2}), and $p$ uniquely determined by the following elliptic problem:
  \begin{equation*}
  \left\{
  \begin{array}{l}
  -\Delta_Hp(\textbf{x}^H,t)=\frac{1}{2h}\int_{-h}^h[\nabla_H\cdot\nabla_H\cdot(v\otimes v)+f_0k\times v]dz,\quad\mbox{in }\Omega, \\
  \int_\Omega p(\textbf{x}^H,t)d\textbf{x}^H=0, \quad p\mbox{ is periodic in }\textbf{x}^H.
  \end{array}
  \right.
  \end{equation*}
\end{definition}

\begin{proposition}
\label{propcont}
Let $v$ be a weak solution to system (\ref{main1})--(\ref{ic}), with
initial data $v_0\in\mathcal H$. Then, there is a subset
$E\subseteq(0,\infty)$ of measure zero, such that
$$
\lim_{t\rightarrow0^+, t\not\in E}v(t)=v_0,\quad\mbox{in }L^2(\Omega).
$$
\end{proposition}

\begin{proof}
  By (iv) of Definition \ref{def1.1}, there is a subset $E\subseteq(0,\infty)$ of measure zero, such that
  $\|v\|_2^2(t)\leq\|v_0\|_2^2,$ for any $t\in(0,\infty)\setminus E$. Thanks to this, and recalling that $v(t)$ is continuous in $L^2_w(\Omega)$ with respect to $t$, it follows from the lower semi-continuity of the norms that
  $$
  \|v_0\|_2^2\leq \varliminf_{t\rightarrow0^+, t\not\in E}\|v\|_2^2(t)\leq \varlimsup_{t\rightarrow0^+, t\not\in E}\|v\|_2^2(t)\leq\|v_0\|_2^2,
  $$
  and thus $\lim_{t\rightarrow0^+,t\not\in E}\|v\|_2^2(t)=\|v_0\|_2^2$. Thanks to this, we deduce
  \begin{align*}
    \lim_{t\rightarrow0^+, t\not\in E}\|v(t)-v_0\|_2^2=& \lim_{t\rightarrow0^+, t\not\in E}\int_\Omega(|v(t)|^2+|v_0|^2-2v(t)\cdot v_0)d\textbf{x}\\
    =&\lim_{t\rightarrow0^+, t\not\in E}\int_\Omega(|v(t)|^2-|v_0|^2)d\textbf{x}=0,
  \end{align*}
  proving the conclusion.
\end{proof}

The integral equality in the following proposition can be viewed as
a replacement of (ii) in Definition \ref{def1.1}, without changing
the definition of weak solutions.

\begin{proposition}\label{appprop1}
  Let $v$ be a weak solution to system (\ref{main1})--(\ref{ic}). Then, for any $0\leq t_1<t_2<\infty$, the following holds
  \begin{align*}
    &\int_\Omega v(\textbf{x},t_2)\varphi(\textbf{x},t_2)d\textbf{x}+\int_{t_1}^{t_2}\int_\Omega \nabla v:\nabla\varphi d\textbf{x}dt\\
    =&\int_\Omega v(\textbf{x},t_1)\varphi(\textbf{x},t_1)d\textbf{x}+\int_{t_1}^{t_2} \int_\Omega[v\cdot\partial_t\varphi -(v\cdot\nabla_H)v\cdot\varphi-w\partial_zv\cdot \varphi]d\textbf{x}dt,
  \end{align*}
  for vector field $\varphi\in C([t_1,t_2]; C^1_\text{per}(\Omega)\cap\mathcal H)\cap C^1([t_1,t_2]; C_\text{per}(\Omega))$.
\end{proposition}

\begin{proof}
We only prove the case that $t_1>0$, the other case that $t_1=0$ can be
proven similarly, by performing the arguments presented below for
time $t_2$ only. Set $h_0=\min\{t_1, t_2-t_1\}$, and define the
extension of $\varphi$ as
\begin{equation*}
  \tilde\varphi(t)=\left\{
  \begin{array}{lr}
    2\varphi(t_2)-\varphi(2t_2-t),&t_2<t\leq t_2+h_0,\\
    \varphi(t),& t_1\leq t\leq t_2,\\
    2\varphi(t_1)-\varphi(2t_1-t),&t_1-h_0\leq t<t_1.
  \end{array}
  \right.
\end{equation*}
One can easily check that $\tilde\varphi\in
C([t_1-h_0,t_2+h_0]; C^1_\text{per}(\Omega)\cap\mathcal H)\cap C^1([t_1-h_0,t_2+h_0]; C_\text{per}(\Omega))$. Let $\chi$ be a function, such that $\chi\in
C^\infty(\mathbb R)$, $\chi\equiv1$ on $(-\infty,0]$, $0\leq\chi\leq1$
and $\chi'\leq0$ on $(0,1/2)$, and $\chi\equiv0$ on $[1/2,\infty)$. For
any $h\in(0,h_0)$, set
$$
\chi_h(t)=\chi\left(\frac{t_1-t}{h}\right)\chi\left(\frac{t-t_2}{h}\right),
\quad t\in\mathbb R.
$$
One can easily verify that $\chi_h\in C_0^\infty((t_1-h,
t_2+h))\subseteq C_0^\infty((0,\infty))$, $\chi_h'\geq0$ on $(t_1-h,
t_1)$, $\chi\equiv1$ on $(t_1,t_2)$, $\chi_h'\leq0$ on $(t_2, t_2+h)$, and
$|\chi_h'|\leq\frac{C}{h}$, for a positive constant $C$ independent of
$h$. Set $\phi=\tilde\varphi\chi_h$. Taking
$\phi$ as the testing function in (ii) of Definition \ref{def1.1}, and
thanks to the properties of $\chi_h$ stated above, one obtains
\begin{align}
\int_{t_1-h_0}^{t_2+h_0}\int_\Omega[-v\cdot\partial_t\tilde
\varphi&+(v\cdot\nabla_H)v\cdot\tilde\varphi+w\partial_zv\cdot\tilde
\varphi +\nabla v:\nabla\tilde\varphi]d\textbf{x}\chi_h(t)dt\nonumber\\
=&\int_{t_1-h}^{t_1}\int_\Omega v\cdot\tilde\varphi
d\textbf{x}\chi_h'(t)dt+ \int_{t_2}^{t_2+h}\int_\Omega v\cdot\tilde\varphi
d\textbf{x}\chi_h'(t)dt, \label{A0}
\end{align}
for any $h\in(0,h_0)$. Denote by LHS the quantity on the left-hand side
of (\ref{A0}). Since $v\in L^\infty(0, t_2+h_0; L^2(\Omega))\cap L^2(0,t_2+h_0; H^1_{\text{per}}(\Omega))$, by the Lebesgue dominate convergence theorem, one can see that
$$
LHS\rightarrow \int_{t_1}^{t_2}\int_\Omega[-v\cdot\partial_t \varphi+(v\cdot\nabla_H)v\cdot\varphi+w\partial_zv\cdot\varphi +\nabla v:\nabla\varphi]d\textbf{x}dt,
$$
as $h\rightarrow0^+$. Define the function $f$ as
$$
f(t)=\int_\Omega v(\textbf{x},t)\tilde\varphi(\textbf{x},t)d\textbf{x}, \quad t\in (t_1-h_0, t_2+h_0).
$$
Recalling that $v\in C([0,\infty); L^2_w(\Omega))$ and noticing that $\tilde\varphi\in C^1(\bar\Omega\times[t_1-h_0, t_2+h_0])$, it is clear that $f\in C([t_1-h_0, t_2+h_0])$. Denote by RHS the quantity on the right-hand side of (\ref{A0}). Recalling that $\chi_h'\geq0$ on $(t_1-h,
t_1)$, and $\chi_h'\leq0$ on $(t_2, t_2+h)$, and applying the mean value theorem of integrals, one deduces
\begin{align*}
RHS=&\int_{t_1-h}^{t_1}f(t)\chi_h'(t)dt+\int_{t_2}^{t_2+h}f(t)\chi_h'(t)dt\\
=&f(t_1-\theta_1(h)h)\int_{t_1-h}^{t_1}\chi_h'(t)dt+f(t_2+\theta_2(h)h)\int_{t_2} ^{t_2+h}\chi_h'(t)dt\\
=&f(t_1-\theta_1(h)h)-f(t_2+\theta_2(h)h)\rightarrow f(t_1)-f(t_2),
\end{align*}
as $h\rightarrow0^+$, where $\theta_1(h),\theta_2(h)\in[0,1]$. Thanks to the above statements, taking $h\rightarrow0^+$ in (\ref{A0}) yields the conclusion.
\end{proof}

Withe the aid of Proposition \ref{appprop1}, one can show that any weak solution to the primitive equations on $\Omega\times(0,\infty)$ is also a weak solution to the primitive equations on $\Omega\times(t,\infty)$, with initial data $v(t)$, for a.e.~$t\in[0,\infty)$, that is the following proposition.

\begin{proposition}
  \label{apppropweak}
  Let $v$ be a global weak solution to system (\ref{main1})--(\ref{ic}). Then, for a.e.~$t\in[0,\infty)$, by viewing time $t$ as the initial time, $v$ is also a weak solution to system (\ref{main1})--(\ref{ic}), on $\Omega\times(t,\infty)$, with initial data $v(t)$.
\end{proposition}

\begin{proof}
  We need to verify the terms (i)--(iv) in Definition \ref{def1.1} on
  the time interval $[t_0,\infty)$, for a.e.~$t_0\in[0,\infty)$. It is
  clear that (i) and (iii) hold, while the validity of (ii) is
  guaranteed by Proposition \ref{appprop1}. We still need to verify that
  (iv) holds on time interval $[t_0,\infty)$, for
  a.e.~$t_0\in[0,\infty)$, or equivalently that
  \begin{equation}
    \frac12\|v\|_2^2(t)+\int_{t_0}^t\|\nabla v\|_2^2(\tau)d\tau\leq\frac12\|\nabla v\|_2^2(t_0),\label{B2}
  \end{equation}
  for a.e.~$t_0\in[0,\infty)$ and a.e.~$t\in(t_0,\infty)$. Setting
  $f(t)=\|v\|_2^2(t)$ and $g(t)=\|\nabla v\|_2^2(t)$, for
  $t\in[0,\infty)$. Then by the regularities of weak solutions, one has
  $f,g\in L^1_{\text{loc}}([0,\infty))$. By (iii) of Definition
  \ref{def1.1}, it holds that
  $$
  \frac12f'(t)+g(t)\leq 0,\quad\mbox{in }\mathcal D'((0,\infty)).
  $$
  Let $0\leq j_\varepsilon=\frac1\varepsilon
  j\left(\frac{t}{\varepsilon}\right)$ be a standard modifier,
  with $j\in C_0^\infty((-1,1))$, and set
  $f_\varepsilon=f*j_\varepsilon$ and $g_\varepsilon=g*j_\varepsilon$,
  for $\varepsilon>0$. For any $t>0$, it is clear that $0\leq
  j_\varepsilon(\,\cdot-t)\in C_0^\infty((0,\infty))$. Thus, one can test the above inequality by $j_\varepsilon(\,\cdot-t)$ to get
  $$
  \frac12f_\varepsilon'(t)+g_\varepsilon(t)\leq0,\quad t\in(0,\infty), \varepsilon\in(0,t).
  $$
  For any $t_0\in(0,\infty)$ and $t\in(t_0,\infty)$, integration the above inequality over the interval $(t_0, t)$ yields
  $$
  \frac12f_\varepsilon(t)+\int_{t_0}^tg_\varepsilon(s)ds\leq\frac12f_\varepsilon(t_0), $$
  for any $\varepsilon\in(0,t_0)$. Note that
  $f_\varepsilon(t)\rightarrow f(t)$ and $g_\varepsilon(t)\rightarrow g(t)$, a.e.\,$t\in(0,\infty)$, and
  $g_\varepsilon\rightarrow g$ in $L^1((0,T))$, for any finite time
  $T$. Thanks to these, taking $\varepsilon\rightarrow0^+$ in the above
  inequality yields
  $$
  \frac12f(t)+\int_{t_0}^tg (s)ds\leq\frac12f(t_0),
  $$
  for a.e.~$t_0\in(0,\infty)$ and a.e.~$t\in(t_0,\infty)$, which is exactly (\ref{B2}). This completes the proof of Proposition \ref{apppropweak}.
\end{proof}

The following proposition states the weak-strong uniqueness result for the primitive equations.

\begin{proposition}
  \label{apppropwsu}
  Given the initial data $v_0\in H^1_{\text{per}}(\Omega)\cap
  \mathcal H$. Let $v_{\text{s}}$ and $v_{\text{w}}$ be the unique
  global strong solution and an arbitrary global weak solution,
  respectively, to system (\ref{main1})--(\ref{ic}), with the same
  initial data $v_0$. Then, we have $v_{\text{s}}\equiv v_{\text{w}}$.
\end{proposition}

\begin{proof}
  Denote $U_{\text{s}}=(v_{\text{s}}, w_{\text{s}})$ and $U_{\text{w}}=(v_{\text{w}}, w_{\text{w}})$, and set $U=(v,w)=U_{\text{w}}-U_{\text{s}}$, where $w_{\text{s}}$ and $w_{\text{w}}$ are determined uniquely by $v_{\text{s}}$ and $v_{\text{w}}$, respectively, through the relation (\ref{main2}).

  Since strong solutions to the primitive equations are smooth away
  from the initial time, one can choose $\varphi=v_{\text{s}}$ as
  the testing function in Proposition \ref{appprop1}, and thus get
  \begin{align*}
    &\int_\Omega v_{\text{w}}(\textbf{x},t)\cdot v_{\text{s}}(\textbf{x},t)d\textbf{x}+\int_s^t\int_\Omega\nabla v_{\text{w}}:\nabla v_{\text{s}}d\textbf{x}d\tau\\
=&\int_\Omega v_{\text{w}}(\textbf{x},s)\cdot v_{\text{s}}(\textbf{x},s)d\textbf{x}+\int_s^t\int_\Omega[v_{\text{w}}\cdot\partial_tv_{\text{s}}-(U_{\text{w}}\cdot \nabla)v_{\text{w}}\cdot v_{\text{s}}]d\textbf{x}d\tau,
  \end{align*}
  for any $0<s<t<\infty$. Adding both sides of the above equality by $\int_s^t\int_\Omega\nabla v_{\text{w}}:\nabla v_{\text{s}}d\textbf{x}d\tau$ yields
  \begin{align*}
    &\int_\Omega v_{\text{w}}(\textbf{x},t)\cdot v_{\text{s}}(\textbf{x},t)d\textbf{x}+2\int_s^t\int_\Omega\nabla v_{\text{w}}:\nabla v_{\text{s}}d\textbf{x}d\tau\nonumber\\
   =&\int_\Omega v_{\text{w}}(\textbf{x},s)\cdot v_{\text{s}}(\textbf{x},s)d\textbf{x}-\int_s^t\int_\Omega(U_{\text{w}}\cdot \nabla)v_{\text{w}}\cdot v_{\text{s}}d\textbf{x}d\tau\nonumber\\
   &+\int_s^t\int_\Omega(v_{\text{w}}\cdot\partial_tv_{\text{s}}+\nabla v_{\text{w}}:\nabla v_{\text{s}})d\textbf{x}d\tau,
  \end{align*}
  for any $0<s<t<\infty$. For the last term on the right-hand side of the above equality, one can integrate by parts and use the equations for $v_{\text{s}}$ to get
  $$
  \int_s^t\int_\Omega(v_{\text{w}}\cdot\partial_tv_{\text{s}}+\nabla v_{\text{w}}:\nabla v_{\text{s}})d\textbf{x}d\tau=-\int_s^t\int_\Omega (U_\text{s}\cdot\nabla)v_{\text{s}}\cdot v_{\text{w}}d\textbf{x}d\tau.
  $$
  Plugging this equality into the previous one leads to
  \begin{align}
    &\int_\Omega v_{\text{w}}(\textbf{x},t)\cdot v_{\text{s}}(\textbf{x},t)d\textbf{x}+2\int_s^t\int_\Omega\nabla v_{\text{w}}:\nabla v_{\text{s}}d\textbf{x}d\tau\nonumber\\
   =&\int_\Omega v_{\text{w}}(\textbf{x},s)\cdot v_{\text{s}}(\textbf{x},s)d\textbf{x}-\int_s^t\int_\Omega[(U_{\text{w}}\cdot \nabla)v_{\text{w}}\cdot v_{\text{s}}+(U_\text{s}\cdot\nabla)v_{\text{s}}\cdot v_{\text{w}}]d\textbf{x}d\tau,\label{A1}
  \end{align}
  for any $0<s<t<\infty$.

  Multiplying the equation for $v_{\text{s}}$ by $v_{\text{s}}$, and
  integrating over $\Omega$, the it follows from integration by parts,
  and integrating over the time interval $(s,t)$ that
  \begin{equation*}
    \frac12\|v_\text{s}\|_2^2(t)+\int_s^t\|\nabla v_{\text{s}}\|_2^2(\tau)d\tau=\frac12\|v_\text{s}\|_2^2(s),
  \end{equation*}
  for any $0<s<t<\infty$. Recalling the energy inequality in Definition \ref{def1.1}, we have
  \begin{align*}
  \frac12\|v_\text{w}\|_2^2(t)+\int_s^t\|\nabla v_{\text{w}}\|_2^2(\tau)d\tau
  \leq \frac12\|v_\text{w}\|_2^2(s)+\frac12(\|v_0\|_2^2- \|v_\text{w}\|_2^2(s)),
  \end{align*}
  for a.e. $0<s<t<\infty$. Adding the above two equalities up, and subtracting the resultant by (\ref{A1}) yield
  \begin{align}
   \frac12\|v\|_2^2(t)+\int_s^t\|\nabla v\|_2^2(\tau)d\tau
  \leq&\int_s^t\int_\Omega[(U_\text{w}\cdot\nabla)v_{\text w}\cdot v_{\text s}+(U_{\text s}\cdot\nabla)v_{\text s}\cdot v_{\text w}] d\textbf{x}d\tau\nonumber\\
  &+\frac12(\|v\|_2^2(s)+\|v_0\|_2^2-\|v_{\text w}\|_2^2(s)),\label{A2}
  \end{align}
  for a.e.\,$0<s<t<\infty$.

  We are going to deal with the first term
  on the right-hand side of (\ref{A2}). Recalling the regularities of $v_\text{s}$ and $v_\text{w}$, it follows
  from integration by parts that
  \begin{align}
    I:=&\int_s^t\int_\Omega[ (U_\text{w} \cdot\nabla) v_{\text w} \cdot v_{\text s}+(U_{\text s}\cdot\nabla) v_{\text s}\cdot v_{\text w}] d\textbf{x}d\tau\nonumber\\
    =&\int_s^t\int_\Omega[(U_\text{w} \cdot\nabla) v_{\text w} \cdot v_{\text s}-(U_{\text s}\cdot\nabla) v_{\text w}\cdot v_{\text s}] d\textbf{x}d\tau\nonumber\\
    =&\int_s^t\int_\Omega(U\cdot\nabla) v_{\text w} \cdot v_{\text s} d\textbf{x}d\tau
    =\int_s^t\int_\Omega(U\cdot\nabla)v\cdot v_{\text s} d\textbf{x}d\tau\nonumber\\
    =&-\int_s^t\int_\Omega (U\cdot\nabla)v_\text{s}\cdot vd\textbf{x}d\tau,\label{A3}
  \end{align}
  for any $0<s<t<\infty$. Note that
  $$
  |\nabla_Hv_\text{s}(\textbf{x}^H,z)|\leq\frac{1}{2h} \int_{-h}^h|\nabla_Hv_{\text s}(\textbf{x}^H,z')|dz'+\int_{-h}^h|\nabla_H\partial_zv_\text{s} (\textbf{x}^H, z')|dz',
  $$
  for any $(\textbf{x}^H,z)\in\Omega$.
  Therefore, by Lemma \ref{lad}, and using the Young inequality, we can estimate $I$ as follows
  \begin{align}
    |I|=&\left|\int_s^t\int_\Omega[(v\cdot\nabla_H)v_\text{s} +w\partial_zv_{\text{s}}]\cdot vd\textbf{x}d\tau\right|\nonumber\\
    \leq&\int_s^t\int_M\int_{-h}^h\left(\frac{|\nabla_Hv_{\text s}|}{2h}+|\nabla_H\partial_z v_{\text s}|\right)dz\int_{-h}^h|v|^2dzd\textbf{x}^Hd\tau\nonumber\\
    &+\int_s^t\int_M\int_{-h}^h|\nabla_Hv|dz\int_{-h}^h|\partial_zv_\text{s}||v|dz d\textbf{x}^Hd\tau\nonumber\\
    \leq&C\int_s^t(\|\nabla_Hv_{\text{s}}\|_2+\|\nabla_H\partial_zv_{\text{s}}\|_2) \|v\|_2(\|v\|_2+\|\nabla_Hv\|_2)d\tau\nonumber\\
    &+C\int_s^t\|\nabla_Hv\|_2\|\partial_zv_{\text s}\|_2^{\frac12}(\|\partial_zv_{\text s}\|_2+\|\nabla_H\partial_zv_{\text s}\|_2)^{\frac12}\nonumber\\
    &\qquad\times\|v\|_2^{\frac12}(\|v\|_2+\|\nabla_Hv\|_2)^{\frac12}d\tau \nonumber\\
    \leq&\frac 12\int_s^t \|\nabla_Hv\|_2^2 d\tau+C\int_s^t(1+\|\partial_zv_{\text s}\|_2^2)(1+\|\nabla v_{\text s}\|_2^2\nonumber\\
    &\qquad+\|\nabla_H\partial_zv_{\text s}\|_2^2)\|v\|_2^2d\tau,\label{A4}
  \end{align}
  for any $0<s<t<\infty$.

  Substituting the above estimate into (\ref{A2}) yields
  \begin{align*}
    \|v\|_2^2(t)+\int_s^t\|\nabla v\|_2^2(\tau)d\tau\leq& C\int_s^t(1+\|\partial_zv_{\text s}\|_2^2)(1+\|\nabla v_{\text s}\|_2^2+\|\nabla_H\partial_zv_{\text s}\|_2^2)\|v\|_2^2d\tau\\
    &+\frac12(\|v\|_2^2(s)+\|v_0\|_2^2-\|v_{\text w}\|_2^2(s)),
  \end{align*}
  for a.e.~$0<s<t<\infty$. By Proposition \ref{propcont}, there is a
  sequence $\{s_n\}$, with $s_n\rightarrow0^+$, as $n\rightarrow\infty$,
  such that both $v_{\text s}(s_n)$ and $v_{\text w}(s_n)$ converge
  to $v_0$ in $L^2(\Omega)$, as $n\rightarrow\infty$.
  Choosing $s=s_n$ in the above inequality, and taking
  $n\rightarrow\infty$ yields
  \begin{align*}
    \|v\|_2^2(t)+\int_s^t\|\nabla v\|_2^2(\tau)d\tau\leq&
    C\int_s^tm(\tau)\|v\|_2^2(\tau)d\tau,
  \end{align*}
  for a.e.~$t\in(0,\infty)$, where
  $$
  m(\tau)=(1+\|\partial_zv_{\text s}\|_2^2(\tau))(1+\|\nabla v_{\text s}\|_2^2(\tau)+\|\nabla_H\partial_zv_{\text s}\|_2^2(\tau)),\quad\tau\in(0,\infty).
  $$
  Noticing that $m\in L^1_{\text{loc}}([0,\infty))$, by the Gronwall inequality, the above inequality implies $\|v\|_2^2(t)=0$, a.e.~$t\in(0,\infty)$, that is $v_{\text s}=v_{\text{w}}$. This completes the proof.
\end{proof}

\begin{corollary}[Regularities for positive time]
\label{corinterior}
Let $v$ be a weak solution to system (\ref{main1})--(\ref{ic}). Then, the following two hold:

(i) $v$ is smooth away from the initial time, i.e.~$v\in C^\infty(\bar\Omega\times(0,\infty))$;

(ii) $v$ satisfies the energy identity
$$
\frac12\|v\|_2^2(t)+\int_0^t\|\nabla v\|_2^2(\tau)d\tau=\frac12\|v_0\|_2^2,
$$
for any $t\in(0,\infty)$.
\end{corollary}

\begin{proof}
(i) By Proposition \ref{apppropweak}, for a.e.~$t\in[0,\infty)$, $v$
is still a weak solution to (\ref{main1})--(\ref{ic}) on
$\Omega\times(t,\infty)$, with initial data $v(t)$, by viewing $t$ as
the initial time. By the energy inequality, it is clear that
$v(\cdot,t)\in H^1_\text{per}\cap\mathcal H$, for a.e.~$t\in(0,\infty)$.
Therefore, there is a subset $E\subseteq(0,\infty)$, of measure zero,
such that $v$ is a weak solution to (\ref{main1})--(\ref{ic}) on
$\Omega\times(t,\infty)$, with initial data $v(t)\in
H^1_\text{per}\cap\mathcal H$, for all $t\in(0,\infty)\setminus E$.
Take an
arbitrary time $t_0\in(0,\infty)$. Thanks to what we just stated, there
is a time $t_0'\in(0,t_0)$, such that $v(\cdot, t_0')\in
H^1_\text{per}(\Omega)\cap\mathcal H$, and $v$ is a weak solution to system
(\ref{main1})--(\ref{ic}), on $\Omega\times(t_0',\infty)$, with initial
data $v(t_0')$. Since $v(t_0')\in H^1_\text{per}(\Omega)\cap\mathcal H$, there is
a unique global strong solution $v_{\text{s}}$ to system
(\ref{main1})--(\ref{ic}), on $\Omega\times(t_0',\infty)$, with initial
data $v(t_0')$. By the weak-strong uniqueness, i.e.\,Proposition \ref{apppropwsu}, we then have $v\equiv
v_\text{s}$, on the time interval $[t_0',\infty)\supseteq[t_0,\infty)$.
Recalling that strong solutions to the primitive equations are smooth
away from the initial time, one has $v=v_{\text s}\in
C^\infty(\bar\Omega\times(t_0,\infty))$, and further
$v\in C^\infty(\bar\Omega\times(0,\infty))$, and thus proves (i).

(ii) Thanks to (i), one can multiply equation (\ref{main1}) by $v$,
and integrating the resultant over $\Omega$, then integration
by parts yields
$$
\frac12\frac{d}{dt}\|v\|_2^2(t)+\|\nabla v\|_2^2(t)=0,
$$
for all $t\in(0,\infty)$. Integrating the above equality over the time interval $(s,t)$ leads to
$$
\frac12\|v\|_2^2(t)+\int_s^t\|\nabla v\|_2^2(\tau)d\tau=\frac12\|v\|_2^2(s),
$$
for any $0<s<t<\infty$. By Proposition \ref{propcont}, there is a
sequence of time $\{s_n\}$, with $s_n\rightarrow0^+$, such that
$v(s_n)\rightarrow v_0$ in $L^2(\Omega)$, as $n\rightarrow\infty$.
Choosing $s=s_n$ in the above equality, and letting $n\rightarrow\infty$
yield the conclusion.
\end{proof}

\section{Proof of Theorem \ref{thmwsu}}
\label{secwsu}

With the preparations in the previous section, section \ref{secinterior}, we are now ready to prove Theorem \ref{thmwsu}.

\begin{proof}[\textbf{Proof of Theorem \ref{thmwsu}}]
Denote by $v_{\text{s}}$ the weak solution stated in Theorem \ref{thmwsu}. By assumption, there is a positive time $T_\text{s}$, such that
\begin{eqnarray*}
  v_{\text s}(x,y,z,t)=\bar v_{\text s}(x,y,z,t)+\bar V_{\text s}(x,y,z,t),\quad (x,y,z)\in\Omega, t\in(0,T_{\text s}),\\
  \partial_z\bar v\in L^\infty(0,T_{\text s}; L^2(\Omega))\cap L^2(0,T_{\text s}; H^1_\text{per}(\Omega)), \quad \bar V_{\text s}\in L^\infty(\Omega\times(0,T_{\text s})).
\end{eqnarray*}
Let $\varepsilon_0$ be a sufficiently small positive number, which will be determined later, and suppose that there is a positive time $T_{\text s}'\in(0, T_{\text s})$, such that
$$
\sup_{0<t<T_{\text s}'}\|\bar V_{\text s}\|_\infty(t)\leq\varepsilon_0.
$$
Take an arbitrary weak solution $v_{\text w}$ solution to system
(\ref{main1})--(\ref{ic}), with the same initial data as $v_{\text s}$.

We denote $U_{\text{s}}=(v_{\text{s}}, w_{\text{s}}),
U_{\text{w}}=(v_{\text{w}}, w_{\text{w}})$ and set
$U=(v,w)=U_{\text{w}}-U_{\text{s}}$, where $w_{\text{s}}$ and
$w_{\text{w}}$ are determined in terms of
$v_{\text{s}}$ and $v_{\text{w}}$,
respectively, through the relation (\ref{main2}).
We are going to show that $v_{\text s}\equiv v_{\text w}$, or
equivalently $v\equiv0$.

By Corollary \ref{corinterior}, any weak solution to system
(\ref{main1})--(\ref{ic}) is smooth away from the initial time.
Therefore, we have $v_{\text s}\in C^\infty(\bar\Omega
\times(0,\infty))$ and $v_{\text w}\in
C^\infty(\bar\Omega\times(0,\infty))$.
By the same argument as that for (\ref{A2}),
we have the following
\begin{align}
   \frac12\|v\|_2^2(t)+\int_s^t\|\nabla v\|_2^2(\tau)d\tau
  \leq&\int_s^t\int_\Omega[(U_\text{w}\cdot\nabla)v_{\text w}\cdot v_{\text s}+(U_{\text s}\cdot\nabla)v_{\text s}\cdot v_{\text w}] d\textbf{x}d\tau\nonumber\\
  &+\frac12(\|v\|_2^2(s)+\|v_0\|_2^2-\|v_{\text w}\|_2^2(s)),\label{B1}
  \end{align}
  for any $0<s<t<\infty$.
  Recalling (\ref{A3}), we have by integration by parts that
  \begin{align*}
    J:=&\int_s^t\int_\Omega[ (U_\text{w} \cdot\nabla) v_{\text w} \cdot v_{\text s}+(U_{\text s}\cdot\nabla) v_{\text s}\cdot v_{\text w}] d\textbf{x}d\tau\nonumber\\
    =&\int_s^t\int_\Omega(U\cdot\nabla)v\cdot v_{\text s}d\textbf{x}d\tau= \int_s^t\int_\Omega(U\cdot\nabla)v\cdot (\bar v_{\text s}+\bar V_{\text s})d\textbf{x}d\tau\nonumber\\
    =&\int_s^t\int_\Omega (U\cdot\nabla)v\cdot\bar V_{\text s}d\textbf{x}d\tau-\int_s^t\int_\Omega (U\cdot\nabla)\bar v_{\text s}\cdot v d\textbf{x}d\tau=:J_1+J_2,
  \end{align*}
  for any $0<s<t<\infty$.

  For $J_1$, by the assumption, it follows from the H\"older and Young inequalities that
  \begin{align*}
    J_1=&\int_s^t\int_\Omega \left[(v\cdot\nabla_H)v-\int_{-h}^z\nabla_H\cdot
    vd\xi\partial_zv\right]\cdot\bar V_{\text s}d\textbf{x}d\tau\\
    \leq&\frac14\int_s^t\|\nabla_Hv\|_2^2d\tau+\int_s^t \|\bar V_{\text s}\|_\infty^2\|v\|_2^2d\tau+C_0\int_s^t\|\bar V_{\text s}\|_\infty\|\nabla v\|_2^2d\tau\\
    \leq&\left(\frac14+C_0\varepsilon_0\right)\int_s^t\|\nabla v\|_2^2+\varepsilon_0^2\int_s^t\|v\|_2^2(\tau)d\tau,
  \end{align*}
  for any $0<s<t\leq T_{\text s}'$, where $C_0$ is a positive constant depending only on $h$. Choosing $\varepsilon_0=\min\left\{1,\frac{1}{8C_0}\right\}$, then one has
  $$
  J_1\leq\frac38\int_s^t\|\nabla v\|_2^2+\int_s^t\|v\|_2^2(\tau)d\tau,
  $$
  for any $0<s<t\leq T_{\text s}'$.
  While for $J_2$, the same argument as that for (\ref{A4}) yields
  \begin{align*}
    J_2
    \leq \frac 18\int_s^t \|\nabla_Hv\|_2^2 d\tau+C\int_s^t(1+\|\partial_z\bar v_{\text s}\|_2^2)
     (1+\|\nabla \bar v_{\text s}\|_2^2+\|\nabla_H\partial_z\bar v_{\text s}\|_2^2)\|v\|_2^2d\tau,
  \end{align*}
  for any $0<s<t<\infty$.

  Thanks to the estimates for $J_1$ and $J_2$, it follows from (\ref{B1}) that
  \begin{align*}
  \|v\|_2^2(t)+\int_s^t\|\nabla v\|_2^2(\tau)d\tau
  \leq&C\int_s^t(1+\|\partial_z\bar v_{\text s}\|_2^2)
     (1+\|\nabla \bar v_{\text s}\|_2^2+\|\nabla_H\partial_z\bar v_{\text s}\|_2^2)\|v\|_2^2d\tau\nonumber\\
  &+(\|v\|_2^2(s)+\|v_0\|_2^2-\|v_{\text w}\|_2^2(s)),
  \end{align*}
  for any $0<s<t\leq T_{\text s}'$. By Proposition \ref{propcont},
  there is a sequence $\{s_n\}$, with $s_n\rightarrow0^+$, as
  $n\rightarrow\infty$, such that both $v_{\text s}(s_n)$ and
  $v_{\text w}(s_n)$ converge to $v_0$ in $L^2(\Omega)$, as
  $n\rightarrow\infty$.
  Choosing $s=s_n$ in the above inequality, and taking
  $n\rightarrow\infty$ lead to
  \begin{align*}
  \|v\|_2^2(t)+\int_0^t\|\nabla v\|_2^2(\tau)d\tau
  \leq&C\int_0^t(1+\|\partial_z\bar v_{\text s}\|_2^2)
     (1+\|\nabla \bar v_{\text s}\|_2^2+\|\nabla_H\partial_z\bar v_{\text s}\|_2^2)\|v\|_2^2d\tau,
  \end{align*}
  for any $0<t\leq T_{\text s}'$.
  By assumption, it is clear that
  $$
  (1+\|\partial_z\bar v_{\text s}\|_2^2(\tau))(1+\|\nabla \bar v_{\text s}\|_2^2(\tau)+\|\nabla_H\partial_z\bar v_{\text s}\|_2^2(\tau))\in L^1((0,T_{\text s})).
  $$
  Therefore, one can apply the Gronwall inequality to the previous
  inequality to conclude that $\|v\|_2^2(t)=0$, for $t\in[0,T_{\text
  s}']$, that is $v_{\text s}\equiv v_{\text w}$, on
  $\Omega\times[0,T_{\text s}']$. Starting from time $T_{\text s}'$,
  both $v_{\text s}$ and $v_{\text w}$ are smooth, and thus are both strong
  solutions to system (\ref{main1})--(\ref{ic}), on
  $\Omega\times(T_{\text s}',\infty)$, with the same initial data at
  time $T_{\text s}'$. By the uniqueness of strong solutions to the
  primitive equations, we have $v_{\text s}\equiv v_{\text w}$, on
  $\Omega\times[T_\text{s}',\infty)$. This completes the proof of
  Theorem \ref{thmwsu}.
\end{proof}

\section{A priori estimates for regular solutions}
\label{sec3}

Let $v_0=\bar v_0+V_0$, with $\bar v_0\in X\cap\mathcal H$ and
$V_0\in L^\infty(\Omega)\cap\mathcal H$. Extend $v_0$, $\bar v_0$ and
$V_0$ periodically to the whole space, and still use the same
notations to denote the relevant extensions. Let $j_\varepsilon$, as before, be a standard compactly supported nonnegative mollifier. Set $\bar v_{0\varepsilon}=\bar v_0*j_\varepsilon$, $V_{0\varepsilon}=V_0*j_\varepsilon$ and $v_{0\varepsilon}=v_0*j_\varepsilon$. It is obvious that $\bar v_{0\varepsilon}\in\mathcal H$, $V_{0\varepsilon}\in\mathcal H$, and $v_{0\varepsilon}=\bar v_{0\varepsilon}+V_{0\varepsilon}\in\mathcal H$. Moreover, we have, for any $q\in[1,\infty]$, and $\varepsilon\in(0,\min\{1,2h\})$,
\begin{eqnarray*}
  &&\|v_{0\varepsilon}\|_q\leq9\|v_0\|_q,\quad\|\bar v_{0\varepsilon}\|_X\leq9\|\bar v_0\|_X,\quad
\|V_{0\varepsilon}\|_\infty\leq\|V_0\|_\infty.
\end{eqnarray*}

Note that $v_{0\varepsilon}\in H^1_\text{per}(\Omega)\cap\mathcal H$,
and consequently, there is a unique global strong solution
$v_\varepsilon$ (cf. \cite{CAOTITI2}),
to system (\ref{main1})--(\ref{ic}), with initial data
$v_{0\varepsilon}$.
Moreover, $v_\varepsilon$ satisfies the basic energy identity and some
additional $L^q$ estimates, which are stated in the following:

\begin{proposition}[\textbf{$L^q$ estimate on $v_\varepsilon$}]
  \label{prop3.1}
Let $v_\varepsilon$ be the unique global strong solution to
system (\ref{main1})--(\ref{ic}), with initial data $v_{0\varepsilon}$.
Then we have the basic energy identity
$$
\frac{1}{2}\|v_\varepsilon\|_2^2(t)+\int_0^t\|\nabla v_\varepsilon\|_2^2(\tau)d\tau=\frac{1}{2}\|v_{0\varepsilon}\|_2^2,
$$
and the $L^4(\Omega)$ estimate
$$
\sup_{0\leq s\leq t}\|v_\varepsilon\|_4(s)\leq \exp\{Ce^{2t}(t+1)(1+\|v_0\|_2^2)^2\}(1+\|v_0\|_4),
$$
for any $t\in[0,\infty)$.
Furthermore, for any $q\in[4,\infty)$, we have
$$
\sup_{0\leq s\leq t}\|v_\varepsilon\|_q(s)\leq \sqrt qK_1(t)(1+\|v_{0}\|_q),
$$
for any $t\in[0,\infty)$, where $C$ is a positive constant depending only on $h$, and $K_1$ is a continuously increasing function on $[0,\infty)$ determined by $h$ and $\|v_0\|_4$.
\end{proposition}

\begin{proof}
Multiplying equation (\ref{main1}) by $v_\varepsilon$ and integrating by parts yields the first conclusion. The second and third ones are direct corollaries of Proposition 3.1 in \cite{CAOLITITI3}. In fact, by (3.9) in \cite{CAOLITITI3}, it follows
\begin{align*}
\sup_{0\leq s\leq t}\|v_\varepsilon\|_4(s)\leq& \exp\{Ce^{2t}(t+1)(1+\|v_{0\varepsilon}\|_2^2)^2\}(1+\|v_{0\varepsilon}\|_4)\\ \leq& \exp\{Ce^{2t}(t+1)(1+\|v_0\|_2^2)^2\}(1+\|v_0\|_4),
\end{align*}
for a positive constant $C$ depending only on $h$,
and by Proposition 3.1 (iii) in \cite{CAOLITITI3}, it follows that
$$
\sup_{0\leq s\leq t}\|v_\varepsilon\|_q(s)\leq \sqrt qK_1(t)(1+\|v_{0\varepsilon}\|_q)\leq \sqrt qK_1(t)(1+\|v_{0 }\|_q),
$$
for some function $K_1$, which is determined by the upper bonds of $\|v_{0\varepsilon}\|_2$ and $\|v_{0\varepsilon}\|_4$; however, since
$\|v_{0\varepsilon}\|_2\leq\|v_0\|_2$ and $\|v_{0\varepsilon}\|_4\leq\|v_0\|_4$, such function can be chosen to be independent of $\varepsilon$.
\end{proof}

Next, we show that away from the initial time, we have the $H^1$ estimates for $v_\varepsilon$.

\begin{proposition}[\textbf{$H^1$ estimate on $v_\varepsilon$}]\label{prop3.2}
Let $v_\varepsilon$ be the unique global solution to system (\ref{main1})--(\ref{ic}), with initial data $v_{0\varepsilon}$. Then, for any $0<t<T<\infty$, we have
$$
\sup_{t\leq s\leq T}\|v_\varepsilon\|_{H^1}^2(s)+\int_t^T(\|\nabla^2 v_\varepsilon\|_2^2+\|\partial_tv_\varepsilon\|_2^2)(s)ds\leq K_2(t,T),
$$
where $K_2$ is a continuous function defined on $\mathbb R_+^2$ and determined only by $h$ and $\|v_0\|_2^2$, where $\mathbb R_+^2=(0,\infty)\times(0,\infty)$.
\end{proposition}

\begin{proof}
Fix $T\in(0,\infty)$, and let $t\in[0,T]$. By Proposition \ref{prop3.1}, we have
$$
\int_0^t\|\nabla v_\varepsilon\|_2^2(s)ds\leq\frac{1}{2}\|v_{0\varepsilon}\|_2^2\leq\frac{1}{2}\|v_0\|_2^2,
$$
and thus, one can choose such $t_0\in(0,t)$ that
$$
\|\nabla v_\varepsilon(t_0)\|_2^2\leq\frac{\|v_0\|_2^2}{t}.
$$
Now, taking $t_0$ as the initial time, then it follows from the $H^1$ type energy estimate for the primitive equations, see (77)--(78) in \cite{CAOTITI2}, we obtain
\begin{align*}
  &\sup_{t\leq s\leq T}\|\nabla v_\varepsilon\|_2^2(s)+\int_t^T\|\nabla^2v_\varepsilon\|_2^2(s)ds\nonumber\\
  \leq&\sup_{t_0\leq s\leq T}\|\nabla v_\varepsilon\|_2^2(s)+\int_{t_0}^T\|\nabla^2v_\varepsilon\|_2^2(s)ds\leq K_2''(t,T),
\end{align*}
where $K_2''$ is a continuous function on $\mathbb R_+^2$ determined only by $h$ and $\|v_0\|_2^2$. Thus, recalling the basic energy identity in Proposition \ref{prop3.1}, we then have
\begin{equation}
  \label{neweq1}
  \sup_{t\leq s\leq T}\|v_\varepsilon\|_{H^1}^2(s) +\int_t^T\|\nabla^2v_\varepsilon\|_2^2(s)ds\leq K_2''(t,T)+\|v_0\|_2^2=:K_2'(t,T).
\end{equation}

Next, we estimate $\|\partial_tv_\varepsilon\|_2^2$. Multiplying equation (\ref{main1}) by $\partial_t v_\varepsilon$, and integrating the resultant over $\Omega$, then it follows from integration by parts, and using (\ref{main1'})--(\ref{main2}) that
\begin{align}
  \label{neweq2}
  \|\partial_tv_\varepsilon\|_2^2=&\int_\Omega(\Delta v_\varepsilon-(v_\varepsilon\cdot\nabla_H)v_\varepsilon-w_\varepsilon \partial_zv_\varepsilon)\cdot\partial_tv_\varepsilon d\textbf{x}\nonumber\\
  \leq&\int_M\left(\int_{-h}^h|\nabla_Hv|dz\right)\left( \int_{-h}^h|\partial_zv_\varepsilon| |\partial_tv_\varepsilon|dz\right)d\textbf{x}^H\nonumber\\
  &+\int_\Omega(|\Delta v_\varepsilon|+|v_\varepsilon||\nabla_Hv_\varepsilon|) |\partial_tv_\varepsilon|d\textbf{x}
  =:L_1+L_2.
\end{align}
The estimates on $L_1$ and $L_2$ are as follows.
By Lemma \ref{lad}, it follows from the Poincar\'e and Young inequalities that
\begin{align*}
  L_1\leq&C\|\nabla_Hv_\varepsilon\|_2^{\frac12}(\|\nabla_Hv_\varepsilon\|_2+ \|\nabla_H^2v_\varepsilon\|_2)^{\frac12}\\
  &\times\|\partial_zv_\varepsilon\|_2^{\frac12} (\|\partial_zv_\varepsilon\|_2+\|\nabla_H\partial_zv_\varepsilon\|_2)^{\frac12} \|\partial_tv_\varepsilon\|_2\\
  \leq&C\|\nabla_Hv_\varepsilon\|_2^{\frac12}\|\nabla_H^2v_\varepsilon\|_2^{\frac12} \|\partial_zv_\varepsilon\|_2^{\frac12}\|\nabla\partial_zv_\varepsilon\|_2^{\frac12} \|\partial_tv_\varepsilon\|_2\\
  \leq&\frac14\|\partial_tv_\varepsilon\|_2^2+C\|\nabla v_\varepsilon\|_2^2\|\nabla^2v_\varepsilon\|_2^2,
\end{align*}
for a positive constant depending only on $h$. For $L_2$, by the H\"older, Sobolev, Poincar\'e and Young inequalities, we deduce
\begin{align*}
  L_2\leq&\|\Delta v_\varepsilon\|_2\|\partial_tv_\varepsilon\|_2+\|v_\varepsilon\|_3 \|\nabla_Hv_\varepsilon\|_6\|\partial_tv_\varepsilon\|_2\\
  \leq&\|\Delta v_\varepsilon\|_2\|\partial_tv_\varepsilon\|_2+ C\|v_\varepsilon\|_{H^1}\|\nabla^2v_\varepsilon\|_2\|\partial_tv _\varepsilon\|_2\\
  \leq&\frac14\|\partial_tv_\varepsilon\|_2^2+C\|\nabla^2 v_\varepsilon\|_2^2(1+\|v_\varepsilon\|_{H^1}^2),
\end{align*}
for a positive constant $C$ depending only on $h$. Substituting the above estimates of $L_1$ and $L_2$ into (\ref{neweq2}), and recalling (\ref{neweq1}), we obtain
\begin{align*}
  \int_t^T\|\partial_tv_\varepsilon\|_2^2(s)ds \leq C\int_t^T(1+\|v_\varepsilon\|_{H^1}^2(s))\|\nabla^2v_\varepsilon\|_2^2(s)ds \leq&C(1+K_2'(t,T))^2.
\end{align*}

Combining the above estimate with (\ref{neweq1}) yields the conclusion, with $K_2(t,T)=K_2'(t,T)+C(1+K_2'(t,T))^2$.
This completes the proof.
\end{proof}

To obtain additional estimates on $v_\varepsilon$, we decompose $v_{\varepsilon}$ into two parts
$$
v_\varepsilon=\bar v_\varepsilon+V_\varepsilon,
$$
such that $V_\varepsilon$ is the unique solution to the following linear system
\begin{eqnarray}
  &\partial_tV_\varepsilon+(v_\varepsilon\cdot\nabla_H)V_\varepsilon+w_\varepsilon\partial_zV_\varepsilon +\nabla_HP_\varepsilon(\textbf{x}^H ,t)-\Delta V_\varepsilon+f_0k\times V_\varepsilon=0,\label{3.5}\\
  &\int_{-h}^h\nabla_H\cdot V_\varepsilon(\textbf{x}^H,z,t)dz=0,\label{3.6}
\end{eqnarray}
subject to the periodic boundary condition, with initial data $V_{0\varepsilon}$.

The solvability of the above linear system can be done in the same way (in fact much easier) as for the primitive equations. Based on the $H^1$ theory for the primitive equations, i.e., the global existence of strong solutions with $H^1$ initial data, one can show that the solution $v_\varepsilon$ to the primitive equations with smooth initial data $v_{0\varepsilon}$ is smooth, and as a result, $V_\varepsilon$ is also smooth.

Moreover, we have the $L^\infty$ estimates on $V_\varepsilon$ stated in the following proposition.

\begin{proposition}[\textbf{$L^\infty$ estimates on $V_\varepsilon$}]
  \label{prop3.3}
Let $V_\varepsilon$ be the unique solution to system (\ref{3.5})--(\ref{3.6}), subject to the periodic boundary condition, with initial data $V_{0\varepsilon}$. Then we have the following estimate
$$
\sup_{0\leq s\leq t}\|V_\varepsilon\|_\infty(s)\leq K_3(t)\|V_0\|_\infty,
$$
for any $t\in[0,\infty)$, where
$$
K_3(t)=C(1+\|v_0\|_4)^{40}(t+1)^2\exp\{Ce^{2t}(t+1)(1+\|v_0\|_4)^4\},
$$
for a positive constant $C$ depending only on $h$.
\end{proposition}

\begin{proof}
Let $q\in[2,\infty)$. Multiplying equation (\ref{3.5}) by $|V_\varepsilon|^{q-2}V_\varepsilon$, and integrating over $\Omega$, it follows from integration by parts that
\begin{equation*}
  \frac{1}{q}\frac{d}{dt}\|V_\varepsilon\|_q^q+\int_\Omega\nabla V_\varepsilon:\nabla(|V_\varepsilon|^{q-2}V_\varepsilon)d\textbf{x}=\int_\Omega P_\varepsilon(\textbf{x}^H,t)\nabla_H\cdot(|V_\varepsilon|^{q-2}V_\varepsilon)d\textbf{x}.
\end{equation*}
Straightforward calculations yield
\begin{equation*}
  \int_\Omega\nabla V_\varepsilon:\nabla(|V_\varepsilon|^{q-2}V_\varepsilon)d\textbf{x}=\int_\Omega|V_\varepsilon|^{q-2}
  (|\nabla V_\varepsilon|^2+(q-2)|\nabla|V_\varepsilon||^2)d\textbf{x}.
\end{equation*}
We thus have
\begin{equation*}
  \frac{1}{q}\frac{d}{dt}\|V_\varepsilon\|_q^q+\left\||V_\varepsilon|^{\frac{q}{2}-1}\nabla V_\varepsilon\right\|_2^2\leq
  \int_\Omega P_\varepsilon(\textbf{x}^H,t)\nabla_H\cdot(|V_\varepsilon|^{q-2}V_\varepsilon)d\textbf{x}.
\end{equation*}
By the H\"older inequality, we deduce
\begin{align*}
  &\int_\Omega P_\varepsilon(\textbf{x}^H,t)\nabla_H\cdot(|V_\varepsilon|^{q-2}V_\varepsilon)d\textbf{x}\\
  \leq&(q-1) \int_M|P_\varepsilon(\textbf{x}^H,t)|\int_{-h}^h|V_\varepsilon|^{q-2}|\nabla_HV_\varepsilon|dzd\textbf{x}^H\\
  \leq&(q-1)\int_M|P_\varepsilon(\textbf{x}^H,t)|\left(\int_{-h}^h|V_\varepsilon|^{q-2}dz\right)^{\frac{1}{2}}\left(\int_{-h}^h |V_\varepsilon|^{q-2}|\nabla_HV_\varepsilon|^2dz \right)^{\frac{1}{2}}d\textbf{x}^H\\
  \leq&(q-1)\left(\int_M|P_\varepsilon|^{\frac{4q}{q+2}}d\textbf{x}^H\right)^{\frac{q+2}{4q}}\left[\int_M\left(\int_{-h}^h|V_\varepsilon |^{q-2}dz\right)^{\frac{2q}{q-2}}d\textbf{x}^H \right]^{\frac{q-2}{4q}}\left\||V_\varepsilon|^{\frac{q}{2}-1} \nabla_HV_\varepsilon\right\|_2\\
  \leq&(q-1)\|P_\varepsilon\|_{\frac{4q}{q+2},M} \left[\int_M\int_{-h}^h|V_\varepsilon|^{2q}dz\left(\int_{-h}^h1dz \right)^{\frac{q+2}{q-2}}d\textbf{x}^H\right]^{\frac{q-2}{4q}} \left\||V_\varepsilon|^{\frac{q}{2}-1}\nabla_HV_\varepsilon\right\|_2 \\
  \leq&(q-1)(2h)^{\frac{q+2}{4q}}\|P_\varepsilon\|_{\frac{4q}{q+2},M}\|V_\varepsilon\|_{2q}^\frac{q-2}{2}\left\|
  |V_\varepsilon|^{\frac{q}{2}-1}\nabla_HV_\varepsilon\right\|_2,
\end{align*}
which, substituted into the previous inequality, yields
\begin{eqnarray}
  \label{3.9}
  &\frac{1}{q}\frac{d}{dt}\|V_\varepsilon\|_q^q+\left\||V_\varepsilon|^{\frac{q}{2}-1}\nabla V_\varepsilon\right\|_2^2\nonumber\\
  &\leq
  (q-1)(2h)^{\frac{q+2}{4q}}\|P_\varepsilon\|_{\frac{4q}{q+2},M}\|V_\varepsilon\|_{2q}^\frac{q-2}{2}\left\|
  |V_\varepsilon|^{\frac{q}{2}-1}\nabla_HV_\varepsilon\right\|_2.
\end{eqnarray}

We need to estimate the term $\|P_\varepsilon\|_{\frac{4q}{q+2},M}$ on the right-hand side of (\ref{3.9}). Applying the operator $\text{div}_H$ to equation (\ref{3.5}), integrating the resulting equation in $z$ over $(-h, h)$, and using (\ref{3.6}), yield
\begin{align*}
  -\Delta_HP_\varepsilon(\textbf{x}^H,t)
  =&\frac{1}{2h}\int_{-h}^h[\nabla_H\cdot\nabla_H\cdot(V_\varepsilon\otimes v_\varepsilon)+f_0\nabla_H\cdot(k\times V_\varepsilon)] dz.
\end{align*}
Note that $P_\varepsilon$ can be uniquely specified by requiring
$\int_MP_\varepsilon d\textbf{x}^H=0$. Decompose $P_\varepsilon=P_\varepsilon^1+P_\varepsilon^2$, where $P_\varepsilon^1$ and $P_\varepsilon^2$ are the unique solutions to
$$
-\Delta_HP_\varepsilon^1(\textbf{x}^H,t)=\frac{1}{2h}\int_{-h}^h \nabla_H\cdot\nabla_H\cdot (V_\varepsilon\otimes v_\varepsilon) dz,
$$
and
$$
-\Delta_HP_\varepsilon^2(\textbf{x}^H,t)= \frac{f_0}{2h}\int_{-h}^h\nabla_H\cdot(k\times V_\varepsilon) dz,
$$
respectively, subject to the periodic conditions, and the average zero condition $\int_MP_\varepsilon^id\textbf{x}^H=0$, for $i=1,2$.
Thus, noticing that
$\frac{4q}{q+2}\in[2,4)$, for $q\in[2,\infty)$, by the elliptic
regularity estimates, the Sobolev, H\"older and Young
inequalities, we deduce
\begin{align*}
  \|P_\varepsilon\|_{\frac{4q}{q+2},M}\leq& \|P_\varepsilon^1\|_{\frac{4q}{q+2}, M}+\|P_\varepsilon^2\|_{\frac{4q}{q+2},M} \leq\|P_\varepsilon^1\|_{\frac{4q}{q+2}, M}+C\|\nabla P_\varepsilon^2\|_{\frac{4q}{3q+2},M}\\
  \leq&C\left\|\int_{-h}^hV_\varepsilon\otimes v_\varepsilon dz\right\|_{\frac{4q}{q+2},M}+C\left\|\int_{-h}^hk\times V_\varepsilon dz\right\|_{\frac{4q}{3q+2},M}\\
  \leq&C\int_{-h}^h\left\|V_\varepsilon\otimes v_\varepsilon\right\|_{\frac{4q}{q+2},M} dz+C\int_{-h}^h\|V_\varepsilon\|_{\frac{4q}{3q+2},M}dz\\
  \leq&C\int_{-h}^h\|v_\varepsilon\|_{4,M}\|V_\varepsilon\|_{2q,M}dz +C\int_{-h}^h\|V_\varepsilon\|_{2q,M}|M|^{\frac34}dz\\
  \leq& C\|v_\varepsilon\|_4\|V_\varepsilon\|_{2q}(2h)^{\frac34-\frac{1}{2q}} +C\|V_\varepsilon\|_{2q}|M|^{\frac34}(2h)^{1-\frac{1}{2q}}\\
  \leq&C(1+2h)(1+\|v_\varepsilon\|_4)\|V_\varepsilon\|_{2q}\leq C(1+\|v_\varepsilon\|_4)\|V_\varepsilon\|_{2q},
\end{align*}
for a positive constant $C$ depending only on $h$.

Substituting the above inequality into (\ref{3.9}) yields
\begin{equation}
  \label{3.10}
  \frac{1}{q}\frac{d}{dt}\|V_\varepsilon\|_q^q+\left\||V_\varepsilon|^{\frac{q}{2}-1}\nabla V_\varepsilon\right\|_2^2\leq Cq(1+\|v_\varepsilon\|_4)\|V_\varepsilon\|_{2q}^\frac{q}{2}\left\|
  |V_\varepsilon|^{\frac{q}{2}-1}\nabla_HV_\varepsilon\right\|_2,
\end{equation}
for a positive constant $C$ depending only on $h$. By the Poincar\'e inequality and the Gagliardo-Nirenberg-Sobolev inequality, $\|\varphi\|_4\leq C\|\varphi\|_1^{\frac{1}{10}}\|\nabla\varphi\|_2^{\frac{9}{10}}$, for any average zero function $\varphi$, we deduce
$$
\|V_\varepsilon\|_{2q}^{\frac{q}{2}}=\left\||V_\varepsilon|^{\frac{q}{2}}\right\|_4\leq C\left\||V_\varepsilon|^{\frac{q}{2}}\right\|_1
^{\frac{1}{10}}\left(\left\||V_\varepsilon|^{\frac{q}{2}}\right\|_1+\left\|\nabla|V_\varepsilon|^{\frac{q}{2}}\right\|_2\right) ^{\frac{9}{10}},
$$
which, substituted into (\ref{3.10}), and using the Young inequality, yields
\begin{align*}
  &\frac{1}{q}\frac{d}{dt}\|V_\varepsilon\|_q^q+\left\||V_\varepsilon|^{\frac{q}{2}-1}\nabla V_\varepsilon\right\|_2^2\\
  \leq& Cq(1+\|v_\varepsilon\|_4)\left\||V_\varepsilon|^{\frac{q}{2}} \right\|_1^{\frac{1}{10}} \left(\left\||V_\varepsilon|^{\frac{q}{2}}\right\|_1+ \left\|\nabla|V_\varepsilon|^{\frac{q}{2}}\right\|_2\right)^{\frac{9}{10}} \left\||V_\varepsilon|^{\frac{q}{2}-1}\nabla_HV_\varepsilon\right\|_2\\
  \leq&Cq(1+\|v_\varepsilon\|_4) \left(\left\||V_\varepsilon|^{\frac{q}{2}}\right\|_1 +q^{\frac{9}{10}}\left\||V_\varepsilon|^{\frac{q}{2}} \right\|_1^{\frac{1}{10}} \left\||V_\varepsilon|^{\frac{q}{2}-1}\nabla V_\varepsilon\right\|_2^{\frac{9}{10}}\right) \left\||V_\varepsilon|^{\frac{q}{2}-1}\nabla_HV_\varepsilon\right\|_2\\
  \leq&\frac{1}{4}\left\||V_\varepsilon|^{\frac{q}{2}-1}\nabla V_\varepsilon\right\|_2^2+Cq^2(1+\|v_\varepsilon\|_4)^2 \left(\left\||V_\varepsilon|^{\frac{q}{2}}\right\|_1^2 +q^{\frac{9}{5}}\left\||V_\varepsilon|^{\frac{q}{2}} \right\|_1^{\frac{1}{5}} \left\||V_\varepsilon|^{\frac{q}{2}-1}\nabla V_\varepsilon\right\|_2^{\frac{9}{5}}\right)\\
  \leq&\frac{1}{2}\left\||V_\varepsilon|^{\frac{q}{2}-1}\nabla V_\varepsilon\right\|_2^2+Cq^2(1+\|v_\varepsilon\|_4)^2\left\||V_\varepsilon|^{\frac{q}{2}}\right\|_1^2
  +Cq^{38}(1+\|v_\varepsilon\|_4)^{20}\left\||V_\varepsilon|^{\frac{q}{2}}\right\|_1^2\\
  \leq&\frac{1}{2}\left\||V_\varepsilon|^{\frac{q}{2}-1}\nabla V_\varepsilon\right\|^2+Cq^{38}(1+\|v_\varepsilon\|_4)^{20}\left(\|V_\varepsilon\|_{\frac{q}{2}}^\frac{q}{2}\right)^2.
\end{align*}
Therefore, applying Proposition \ref{prop3.1}, we have
\begin{align*}
  \frac{d}{dt}\|V_\varepsilon\|_q^q\leq&Cq^{40}(1+\|v_\varepsilon\|_4)^{20} \left(\|V_\varepsilon\|_{\frac{q}{2}}^\frac{q}{2}\right)^2\\
  \leq&Cq^{40}\left[(1+\|v_0\|_4)\exp\{Ce^{2t}(t+1)(1+\|v_0\|_2^2)^2\} \right]^{20}
  \left(\|V_\varepsilon\|_{\frac{q}{2}}^\frac{q}{2}\right)^2\\
  \leq&Cq^{40}\left[(1+\|v_0\|_4)\exp\{Ce^{2t}(t+1)(1+\|v_0\|_4)^4\} \right]^{20}
  \left(\|V_\varepsilon\|_{\frac{q}{2}}^\frac{q}{2}\right)^2\\
  =&Cq^{40}S_0(t)\left(\|V_\varepsilon\|_{\frac{q}{2}}^\frac{q}{2}\right)^2,
\end{align*}
with $K_0$ given by
$$
S_0(t)=[(1+\|v_0\|_4)\exp\{Ce^{2t}(t+1)(1+\|v_0\|_4)^4\}]^{20},
$$
and $C$ a positive constant depending only on $h$.
Set $\delta_0=\|V_0\|_\infty$. Recalling that $\|V_{0\varepsilon}\|_\infty\leq\|V_0\|_\infty$, it follows that $\|V_{0\varepsilon}\|_q^q\leq 2h\delta_0^q$. On account of this, it follows from the previous inequality that
\begin{equation}\label{neweq3}
\sup_{0\leq s\leq t}\|V_\varepsilon\|_q^q\leq 2h\delta_0^q+C_0^* q^{40}S_1(t)\left(\sup_{0\leq s\leq t}\|V_\varepsilon\|_{\frac{q}{2}}^{\frac{q}{2}}\right)^2,
\end{equation}
for a constant $C_0^*$ depending only on $h$, where $S_1(t):=\int_0^tS_0(s) ds$.

Define
$$
A_k(t)=\sup_{0\leq s\leq t}\|V_\varepsilon\|_{2^k}^{2^k},\quad k=1,2,\cdots.
$$
Then, setting $q=2^{k+1}$ in (\ref{neweq3}) yields
$$
A_{k+1}(t)\leq 2h\delta_0^{2^{k+1}}+C_0^*2^{40(k+1)}S_1(t)A_k(t)^2,\quad k=1,2,\cdots,
$$
where $C_0^*$ is a positive constant depending only on $h$. Setting
$$
M_0(t)=2h+(1+C_0^*)2^{80}(1+S_1(t)),
$$
then we have
$$
C_0^*2^{40(k+1)}S_1(t)\leq(1+C_0^*)2^{80k}(1+S_1(t)) \leq[(1+C_0^*)2^{80}(1+S_1(t))]^k\leq M_0(t)^k,
$$
and consequently
$$
A_{k+1}(t)\leq M_0(t)\delta_0^{2^{k+1}}+M_0(t)^kA_k(t)^2,\quad k=1,2,\cdots.
$$
It is obviously that $M_0(t)\geq2$. Thus, we can apply Lemma \ref{iterate} to deduce that
$$
A_k(t)\leq (M_0(t)^4\delta_0^2)^{{2^{k-1}}}, \quad k=1,2,\cdots.
$$
Recalling the definition of $A_k$, then for any $s\in[0,t]$, we have
$$
\|V_\varepsilon\|_{2^k}^{2^k}(s)\leq A_k(t)\leq(M_0(t)^4\delta_0^2)^{2^{k-1}},
$$
which implies $
\|V_\varepsilon\|_{2^k}(s)\leq M_0(t)^2\delta_0$, for any positive integer $k$ and every $s\in[0,t]$. Taking $k\rightarrow\infty$, we conclude
$$
\|V_\varepsilon\|_\infty(s)\leq M_0(t)^2\delta_0=M_0(t)^2\|V_0\|_\infty,
$$
for every $s\in[0,t]$, and hence
$$
\sup_{0\leq s\leq t}\|V_\varepsilon\|_\infty(s)\leq M_0(t)^2\|V_0\|_\infty.
$$
Recalling the expressions of $M_0$ and $S_0$, one can easily verify that
$$
M_0(t)^2\leq C(1+\|v_0\|_4)^{40}(t+1)^2\exp\{e^{2t}(t+1)(1+\|v_0\|_4)^4\}=:K_3(t),
$$
for a positive constant $C$, depending only on $h$. Therefore, the conclusion holds.
\end{proof}

Recalling the definition of $V_\varepsilon$, it is then clear that $\bar v_\varepsilon$ satisfies
\begin{eqnarray}
&\partial_t\bar v_\varepsilon+(v_\varepsilon\cdot\nabla_H)\bar v_\varepsilon+w_\varepsilon\partial_z\bar v_\varepsilon+\nabla_H\bar p_\varepsilon(\textbf{x}^H,t)-\Delta \bar v_\varepsilon+f_0k\times
\bar v_\varepsilon=0,\label{3.1}\\
&\int_{-h}^h\nabla_H\cdot \bar v_\varepsilon(\textbf{x}^H,z,t)dz=0,\label{3.2}
\end{eqnarray}
subject to the periodic boundary condition, with initial data $\bar v_{0\varepsilon}$. Moreover, since $v_\varepsilon$ and $ V_\varepsilon$ are both smooth, so is $\bar v_\varepsilon$.

For $\bar v_\varepsilon$, we have the following proposition, which states the estimate of $\bar v_\varepsilon$ in $X$.

\begin{proposition}[\textbf{Estimate on $\bar v_\varepsilon$ in $X$}]
  \label{prop3.4}
Let $\bar v_\varepsilon$ be the unique solution to system (\ref{3.1})--(\ref{3.2}), subject to periodic boundary condition, with initial data $\bar v_{0\varepsilon}$. Then we have the following estimate
$$
\sup_{0\leq s\leq t}(\|\bar v_\varepsilon\|_2^2+\|\partial_z\bar v_\varepsilon\|_2^2) +\int_0^t(\|\nabla \bar v_\varepsilon\|_2^2+\|\nabla\partial_z\bar v_\varepsilon\|_2^2)d\tau\leq K_4(t),
$$
where $K_4$ is a continuous function on $[0,\infty)$, determined only by $h$, and the initial norms $\|\bar v_0\|_2, \|V_0\|_\infty$ and $\|v_0\|_6$.
\end{proposition}

\begin{proof}
Multiplying equation (\ref{3.1}) by $\bar v_\varepsilon$, and integrating over $\Omega$, then it follows from integration by parts that
\begin{equation*}
  \frac{1}{2}\frac{d}{dt}\|\bar v_\varepsilon\|_2^2+\|\nabla\bar v_\varepsilon\|_2^2=0,
\end{equation*}
which implies
\begin{equation}
  \frac{1}{2}\sup_{0\leq s\leq t}\|\bar v_\varepsilon\|_2^2+\int_0^t\|\nabla\bar v_\varepsilon\|_2^2d\tau\leq\frac{1}{2}\|\bar v_{0\varepsilon}\|_2^2\leq\frac{1}{2}\|\bar v_0\|_2^2. \label{3.1-1}
\end{equation}

Set $\bar u_\varepsilon=\partial_z\bar v_\varepsilon$. Differentiating equation (\ref{3.1}) with respect to $z$ yields
\begin{equation*}
  \partial_z\bar u_\varepsilon+(v_\varepsilon\cdot\nabla_H)\bar u_\varepsilon+w_\varepsilon\partial_z\bar u_\varepsilon-\Delta\bar u_\varepsilon+(\partial_zv_\varepsilon\cdot\nabla_H)\bar v_\varepsilon-(\nabla_H\cdot v_\varepsilon)\bar u_\varepsilon+f_0k\times\bar u_\varepsilon=0.
\end{equation*}
Multiplying the above equation by $\bar u_\varepsilon$, and integrating over $\Omega$, then it follows from integration by parts that
\begin{equation}
  \frac{1}{2}\frac{d}{dt}\|\bar u_\varepsilon\|_2^2+\|\nabla\bar u_\varepsilon\|_2^2=\int_\Omega[(\nabla_H\cdot v_\varepsilon)\bar u_\varepsilon-(\partial_z v_\varepsilon\cdot\nabla_H)\bar v_\varepsilon]\cdot\bar u_\varepsilon d\textbf{x}. \label{3.12}
\end{equation}

We are going to estimate the terms on the right-hand side of the above equality. For the first term, integration by parts, and using the H\"older, Sobolev, Poincar\'e and Young inequalities lead to
\begin{align*}
  I_1:=&\int_\Omega(\nabla_H\cdot v_\varepsilon)|\bar u_\varepsilon|^2d\textbf{x}=-2\int_\Omega(v_\varepsilon\cdot\nabla_H)\bar u_\varepsilon\cdot\bar u_\varepsilon d\textbf{x}\\
  \leq&2\|v_\varepsilon\|_6\|\nabla_H\bar u_\varepsilon\|_2\|\bar u_\varepsilon\|_3\leq C\|v_\varepsilon\|_6\|\nabla_H\bar u_\varepsilon\|_2\|\bar u_\varepsilon\|_2^{\frac{1}{2}}\|\nabla_H\bar u_\varepsilon\|_2^{\frac{1}{2}}\\
  \leq&\frac{1}{6}\|\nabla\bar u_\varepsilon\|_2^2+C\|v_\varepsilon\|_6^4\|\bar u_\varepsilon\|_2^2\leq\frac{1}{6}\|\nabla\bar u_\varepsilon\|_2^2+C\|v_\varepsilon\|_6^4\|\nabla\bar v_\varepsilon\|_2^2.
\end{align*}
For the second term, it follows from integration by parts that
\begin{align*}
  I_2:=&-\int_\Omega(\partial_zv_\varepsilon\cdot\nabla_H)\bar v_\varepsilon\cdot\bar u_\varepsilon d\textbf{x}=-\int_\Omega\partial_z (\bar v_\varepsilon+V_\varepsilon)\cdot\nabla_H\bar v_\varepsilon \cdot\bar u_\varepsilon d\textbf{x}\\
  =&-\int_\Omega(\bar u_\varepsilon\cdot\nabla_H)\bar v_\varepsilon\cdot\bar u_\varepsilon d\textbf{x}-\int_\Omega(\partial_zV_\varepsilon\cdot\nabla_H)\bar v_\varepsilon\cdot\bar u_\varepsilon d\textbf{x}=:I_{21}+I_{22}.
\end{align*}

Estimates on $I_{21}$ and $I_{22}$ are given as follows. Integrating by parts, it follows from the H\"older, Sobolev, Poincar\'e and Young inequalities that
\begin{align*}
  I_{21}=&-\int_\Omega(\bar u_\varepsilon\cdot\nabla_H)\bar v_\varepsilon\cdot \bar u_\varepsilon d\textbf{x}\\
  =&\int_\Omega[(\nabla_H\cdot\bar u_\varepsilon)\bar v_\varepsilon\cdot\bar u_\varepsilon+(\bar u_\varepsilon\cdot\nabla_H) \bar u_\varepsilon\cdot\bar v_\varepsilon]d\textbf{x}\\
  \leq&2\|\nabla_H\bar u_\varepsilon\|_2\|\bar v_\varepsilon\|_6\|\bar u_\varepsilon\|_3\leq C\|\nabla_H\bar u_\varepsilon\|_2\|\bar v_\varepsilon\|_6\|\bar u_\varepsilon\|_2^{\frac{1}{2}}\|\nabla\bar u_\varepsilon\|_2^{\frac{1}{2}}\\
  \leq&
\frac{1}{6}\|\nabla\bar u_\varepsilon\|_2^2+C\|\bar v_\varepsilon\|_6^4\|\bar u_\varepsilon\|_2^2
  \leq \frac{1}{6}\|\nabla\bar u_\varepsilon\|_2^2+C( \|v_\varepsilon\|_6^4+\|V_\varepsilon\|_6^4)\|\bar u_\varepsilon\|_2^2 \\
  \leq&\frac{1}{6}\|\nabla\bar u_\varepsilon\|_2^2+C(\|v_\varepsilon\|_6^4+\|V_\varepsilon\|_\infty^4)\|\nabla\bar v_\varepsilon\|_2^2,
\end{align*}
and
\begin{align*}
  I_{22}=&-\int_\Omega(\partial_z V_\varepsilon\cdot\nabla_H)\bar v_\varepsilon\cdot \bar u_\varepsilon d\textbf{x}\\
  =&\int_\Omega[(V_\varepsilon\cdot\nabla_H)\bar u_\varepsilon\cdot\bar u_\varepsilon+(V_\varepsilon\cdot\nabla_H)\bar v_\varepsilon\cdot\partial_z\bar u_\varepsilon]d\textbf{x}\\
  \leq&\|V_\varepsilon\|_\infty\|\nabla_H\bar u_\varepsilon\|_2\|\bar u_\varepsilon\|_2+\|V_\varepsilon\|_\infty\|\nabla_H\bar v_\varepsilon\|_2\|\partial_z\bar u_\varepsilon\|_2\\
  \leq&\frac{1}{6}\|\nabla\bar u_\varepsilon\|_2^2+C\|V_\varepsilon\|_\infty^2(\|\bar u_\varepsilon\|_2^2+\|\nabla_H\bar v_\varepsilon\|_2^2)\\
  \leq&\frac{1}{6}\|\nabla_H\bar u_\varepsilon\|_2^2+C\|V_\varepsilon\|_\infty^2\|\nabla \bar v_\varepsilon\|_2^2.
\end{align*}

Substituting the estimates for $I_1, I_{21}$ and $I_{22}$ into (\ref{3.12}), we have
$$
\frac{d}{dt}\|\bar u_\varepsilon\|_2^2+\|\nabla\bar u_\varepsilon\|_2^2\leq C(\|v_\varepsilon\|_6^4+\|V_\varepsilon\|_\infty^4+1)\|\nabla\bar v_\varepsilon\|_2^2,
$$
for a constant $C$ depending only on $h$. Integrating the above inequality with respect to time, recalling (\ref{3.1-1}), and applying Proposition \ref{prop3.1} and Proposition \ref{prop3.3}, we have
\begin{align*}
  \sup_{0\leq s\leq t}&\|\bar u_\varepsilon\|_2^2(s)+\int_0^t\|\nabla\bar u_\varepsilon\|_2^2(\tau)d\tau\leq C\sup_{0\leq s\leq t}(\|v_\varepsilon\|_6^4+\|V_\varepsilon\|_\infty^4+1)\int_0^t\|\nabla\bar v_\varepsilon\|_2^2d\tau\\
  \leq&C\|\bar v_0\|_2^2[K_1(t)^4(1+\|v_0\|_6^4)+1+K_3(t)^4\|V_0\|_\infty^4]=:K_4'(t).
\end{align*}
Combining the above inequality with (\ref{3.1-1}), we then obtain
$$
\sup_{0\leq s\leq t}(\|\bar v_\varepsilon\|_2^2+\|\bar u_\varepsilon\|_2^2)
+\int_0^t(\|\nabla\bar v_\varepsilon\|_2^2+\|\nabla\bar u_\varepsilon\|_2^2)d\tau\leq K_4'(t)+\|\bar v_0\|_2^2=:K_4(t),
$$
completing the proof.
\end{proof}

\section{Proof of Theorem \ref{thmain}}
\label{sec4}

Based on the results established in the previous section, we can now give the proof of Theorem \ref{thmain} as follows.

\begin{proof}[\textbf{Proof of Theorem \ref{thmain}}]
(i) Suppose $v_0=\bar v_0+V_0$, with
$\bar v_0\in X\cap\mathcal H$ and $V_0\in L^\infty(\Omega)
\cap\mathcal H$.
Extend $v_0, \bar v_0$ and $V_0$ periodically to the whole space,
and use the same notations to the relevant extensions.
Let $j_\varepsilon$, as before, be the standard compactly supported nonnegative mollifier, and set
$\bar v_{0\varepsilon}=\bar v_0*j_\varepsilon$,
$V_{0\varepsilon}=V_0*j_\varepsilon$ and
$v_{0\varepsilon}=v_0*j_\varepsilon$. It is then obvious that $\bar
v_{0\varepsilon}\in X\cap\mathcal H$, $V_{0\varepsilon}\in
L^\infty(\Omega)\cap\mathcal H$ and $v_{0\varepsilon}=\bar
v_{0\varepsilon}+V_{0\varepsilon}$.

Let $v_\varepsilon$ be the unique global strong solution to system (\ref{main1})--(\ref{ic}), with initial data $v_{0\varepsilon}$. As in section \ref{sec3}, we decompose $v_\varepsilon$ into two parts as
$$
v_\varepsilon=V_\varepsilon+\bar v_\varepsilon,
$$
where $V_\varepsilon$ is the unique solution to system (\ref{3.1})--(\ref{3.2}), subject to periodic boundary conditions, with initial data $V_{0\varepsilon}$. By Propositions \ref{prop3.1}--\ref{prop3.4}, we have the estimates
\begin{eqnarray*}
  &\displaystyle\frac{1}{2}\|v_\varepsilon\|_2^2(t)+\int_0^t\|\nabla v_\varepsilon\|_2^2(\tau)d\tau\leq\frac{1}{2}\|v_0\|_2^2,\\
  &\displaystyle\sup_{0\leq s\leq t}\|v_\varepsilon\|_6(s) \leq \sqrt6K_1(t)(1+\|v_{0}\|_6),\\
  &\displaystyle\sup_{t\leq s\leq T}\| v_\varepsilon\|_{H^1}^2(s)+\int_t^T\|(\nabla^2v_\varepsilon, \partial_tv_\varepsilon)\|_2^2(s) \leq K_2(t,T),\\
  &\displaystyle\sup_{0\leq s\leq t}\|V_\varepsilon\|_\infty(s)\leq   K_3(t)\|V_0\|_\infty,\\
  &\displaystyle\sup_{0\leq s\leq t}\|(\bar v_\varepsilon,\partial_z\bar v_\varepsilon)\|_2^2(s)+\int_0^t\|(\nabla\bar v_\varepsilon,\nabla\partial_z\bar v_\varepsilon)\|_2^2(s)ds \leq K_4(t),
\end{eqnarray*}
for any $0<t<T<\infty$, where $K_i, i=1,2,3,4$, are the same functions as those in Propositions \ref{prop3.1}--\ref{prop3.4}.

Note that the associated pressure $p_\varepsilon(\textbf{x}^H,t)$ for the solution $v_\varepsilon$ in system (\ref{main1})--(\ref{main2}) can be decomposed as $p_\varepsilon=p_\varepsilon^1+p_\varepsilon^2$, where $p_\varepsilon^1$ and $p_\varepsilon^2$ are the unique solutions to systems
$$
-\Delta_Hp_\varepsilon^1(\textbf{x}^H,t)=\frac{1}{2h}\nabla_H\cdot \nabla_H\cdot \int_{-h}^h(v_\varepsilon\otimes v_\varepsilon)dz,
$$
and
$$
-\Delta_Hp_\varepsilon^2(\textbf{x}^H,t)= \frac{ f_0}{2h}\int_{-h}^h\nabla_H\cdot (k\times v_\varepsilon) dz,
$$
respectively, subject to the periodic boundary conditions, and the
average zero condition
$\int_Mp_\varepsilon^id\textbf{x}^H=0$, for $i=1,2$. As
a result, by the elliptic regularity estimates, and recalling the
$L^\infty(0,t; L^6(\Omega))$ estimate on $v_\varepsilon$ stated above,
we obtain by the H\"older inequality that $$
\|p_\varepsilon^1\|_{L^\infty(0,t;L^3(\Omega))}\leq
C\|v_\varepsilon\|_{L^\infty(0,t; L^6(\Omega))}^2\leq
CK_1(t)^2(1+\|v_0\|_6^2),
$$
and
$$
\|\nabla p_\varepsilon^2\|_{L^\infty(0,t; L^2(\Omega))}\leq C\|v_\varepsilon\|_{L^\infty(0,t; L^2(\Omega))}\leq C\|v_0\|_2^2,
$$
for any $t\in(0,\infty)$.

Rewrite equation (\ref{main1}) for $v_\varepsilon$ as
$$
\partial_tv_\varepsilon+\nabla_H\cdot(v_\varepsilon\otimes v_\varepsilon)+\partial_z(w_\varepsilon v_\varepsilon)+\nabla_Hp_\varepsilon(\textbf{x}^H,t)-\Delta v_\varepsilon+f_0k\times v_\varepsilon=0.
$$
Noticing that $L^2(\Omega)\hookrightarrow L^{\frac54}(\Omega)\hookrightarrow W^{-1,\frac{5}{4}}(\Omega)$, and recalling that $w_\varepsilon$ is determined by $v_\varepsilon$ through (\ref{main2}), it follows from the above equation and using the H\"older inequality that
\begin{align*}
  \|\partial_tv_\varepsilon\|_{W_\text{per}^{-1,\frac{5}{4}}}\leq& \|\nabla_H\cdot(v_\varepsilon\otimes v_\varepsilon)\|_{W_\text{per}^{-1,\frac{5}{4}}} +\|\partial_z(w_\varepsilon v_\varepsilon)\|_{W_\text{per}^{-1,\frac{5}{4}}}+\|\Delta v_\varepsilon\|_{W_\text{per}^{-1,\frac{5}{4}}}\\
  &+\|\nabla_Hp^1_\varepsilon \|_{W_\text{per}^{-1,\frac{5}{4}}} +\|\nabla_Hp_\varepsilon^2\|_{W_\text{per}^{-1,\frac{5}{4}}} +f_0\|k\times v_\varepsilon\|_{W_\text{per}^{-1,\frac{5}{4}}}\\
  \leq& C(\|v_\varepsilon\otimes v_\varepsilon\|_{ {\frac{5}{4}}}+\|w_\varepsilon v_\varepsilon\|_{ {\frac{5}{4}}}+\|\nabla v_\varepsilon\|_{ {\frac{5}{4}}}+\|p_\varepsilon^1\|_{ {\frac{5}{4}}}+\|\nabla_H p_\varepsilon^2\|_2+\|v_\varepsilon\|_{ 2 })\\
  \leq&C(\|v_\varepsilon\|_{ \frac{5}{2}}^2+\|w_\varepsilon\|_2\|v_\varepsilon\|_{\frac{10}{3}} +\|\nabla v_\varepsilon\|_2+\|p_\varepsilon^1\|_2+\|v_\varepsilon\|_2)\\
  \leq&C(\|v_\varepsilon\|_6^2+\|\nabla v_\varepsilon\|_2\|v_\varepsilon\|_6+\|\nabla v_\varepsilon\|_2+\|v_\varepsilon\|_2),
\end{align*}
and thus we have
\begin{align*}
  \|\partial_tv_\varepsilon\|_{L^{\frac{5}{4}}(0,t;W_\text{per} ^{-1,\frac{5}{4}})} \leq& C(\|v_\varepsilon\|_{L^\infty(0,t;L^6(\Omega))}^2 +\|v_\varepsilon\|_{L^\infty (0,t;L^6(\Omega))}\|\nabla v_\varepsilon\|_{L^2(\Omega\times(0,t))}\\
  &+\|\nabla v_\varepsilon\|_{L^2(\Omega\times(0,t))}+\|v_\varepsilon\|_{L^\infty(0,t; L^2(\Omega))})\\
  \leq&C(\|v_\varepsilon\|_{L^\infty(0,t;L^6(\Omega))}^2+\|\nabla v_\varepsilon\|_{L^2(\Omega\times(0,t))}^2+1)\leq B(t),
\end{align*}
for a finite valued function $B(t)$, for $t\in[0,T]$, which is independent of $\varepsilon$.

On account of the estimates obtained above, and noticing that
$H^1_\text{per}\hookrightarrow\hookrightarrow L^2\hookrightarrow W_\text{per}^{-1,\frac{5}{4}}$, and that $H_\text{per}^2\hookrightarrow\hookrightarrow H_\text{per}^1\hookrightarrow L^2$, one can apply the Aubin--Lions lemma, i.e., Lemma \ref{AL}, to deduce that there is a subsequence, still denoted by $\{v_\varepsilon\}$, and a vector field $v=\bar v+V$, such that
\begin{eqnarray*}
  &v_\varepsilon\rightarrow v\mbox{ in }L^2(\Omega\times(0,t))\cap C([0,t];W_\text{per}^{-1,\frac54}(\Omega)),\\
   &v_\varepsilon\rightharpoonup v\mbox{ in }L^2(0,t; H_\text{per}^1(\Omega)\cap\mathcal H),
  \quad v_\varepsilon\overset{*}{\rightharpoonup}v\mbox{ in }L^\infty(0,t;L^2(\Omega)),\\
  &\nabla v_\varepsilon\overset{*}{\rightharpoonup}\nabla v\mbox{ in }L^\infty(t,T;L^2(\Omega)),
  \quad (\nabla^2v_\varepsilon,\partial_tv_\varepsilon) \rightharpoonup(\nabla^2 v,\partial_tv)\mbox{ in }L^2(\Omega\times(t,T)),\\
  &V_\varepsilon\overset{*}{\rightharpoonup}V\mbox{ in }L^\infty(\Omega\times(0,t)),\\
  &\partial_z\bar v_\varepsilon\overset{*}{\rightharpoonup}\partial_z\bar v\mbox{ in }L^\infty(0,t;L^2(\Omega)),\quad\partial_z\bar v_\varepsilon\rightharpoonup\partial_z\bar v\mbox{ in }L^2(0,t;H_\text{per}^1(\Omega)),
\end{eqnarray*}
for any $0<t<T<\infty$, where $\rightharpoonup$ and
$\overset{*}{\rightharpoonup}$ are the weak and weak-* convergences,
 respectively.
Moreover, by the weakly lower semi-continuity of the relevant norms, we have
\begin{equation}
  \label{4.1}
  \|V\|_{L^\infty(\Omega\times(0,t))}\leq K_3(t)\|V_0\|_\infty,\quad\frac{1}{2}\|v\|_2^2(t)+\int_0^t\|\nabla v\|_2^2(\tau)d\tau\leq\frac{1}{2}\|v_0\|_2^2,
\end{equation}
for a.e.~$t\in[0,\infty)$.

We claim that $v$ is a weak solution to system (\ref{main1})--(\ref{ic}). To this end, we need to verify (i)--(iv) in Definition \ref{def1.1}. By the strong convergences stated above, one can see that
$$
v\in C([0,\infty); W_\text{per}^{-1,\frac54}(\Omega))\cap L^\infty_{\text{loc}}([0,\infty); L^2(\Omega))\cap L^2_{\text{loc}}([0,\infty); H_\text{per}^1(\Omega)\cap\mathcal H),
$$
from which, by the density argument, one obtains
$$
v\in C([0,\infty);L^2_w(\Omega))\cap L^2_{\text{loc}}([0,\infty); H_\text{per}^1(\Omega)\cap\mathcal H),
$$
verifying (i) in Definition \ref{def1.1}.
Recalling that $v_\varepsilon$ is smooth, it is clear that
$v_\varepsilon$ satisfies the weak formula, i.e.~(ii),
stated in Definition \ref{def1.1}. As a result, recalling the
convergences stated in the previous paragraph, one can take the limit,
as $\varepsilon$ goes to zero, to show that $v$ satisfies the weak formula (ii) in Definition \ref{def1.1}. Note that
$$
v\in L^2(t,T; H_\text{per}^1(\Omega)),\quad \partial_tv\in L^2(t,T; L^2(\Omega)),
$$
$v$ is actually a strong solution to the primitive equations, away from the initial time. By the aid of this fact, the differential energy inequality (iii) in Definition \ref{def1.1} actually holds as an equality. The term (iv) is guaranteed by (\ref{4.1}). Therefore, $v$ is a weak solution to system (\ref{main1})--(\ref{ic}), with initial data $v_0$.

Recall that we have the decomposition $v=\bar v+V$, with $\bar v$ and $V$ being the weak limits of $\bar v_\varepsilon$ and $V_\varepsilon$, respectively. By the weakly lower semi-continuity of norms, $\bar v$ and $V$ have the same regularities and same estimates to those for $\bar v_\varepsilon$ and $V_\varepsilon$, respectively. This completes the proof of (i).

(ii) Let $v$ be the weak solution established in (i). Note that $\mu$ is continuous on $[0,\infty)$. There is a positive short time $T_v$, such that
$$
\sup_{0\leq t\leq T_v}\|V\|_\infty(t)\leq\mu(T_v)\|V_0\|_\infty\leq 2\mu(0)\|V_0\|_\infty\leq\varepsilon_0.
$$
Thanks to this, by Theorem \ref{thmwsu}, $v$ is the unique weak solution to system (\ref{main1})--(\ref{ic}), with initial data $v_0$. This completes the proof of (ii).
\end{proof}

\section*{Acknowledgments}
{E.S.T. is thankful to  the kind hospitality of the Universidade Federal do Rio de Janeiro (UFRJ) and Instituto Nacional de Matem\' {a}tica  Pura  e Aplicada (IMPA) where part of this work was completed; and to the partial support of  the  CNPq-CsF grant \# 401615/2012-0, through the program Ci\^encia sem Fronteiras. This work was supported in part by the ONR grant N00014-15-1-2333 and the NSF grants DMS-1109640 and DMS-1109645}
\par

\end{document}